\documentclass{article}

\usepackage[utf8]{inputenc}
\usepackage[fleqn]{amsmath} 
\usepackage{amsfonts}
\usepackage{amssymb}
\usepackage{bm}
\usepackage{dsfont}
\usepackage{scrextend}
\usepackage{amsthm}
\usepackage{graphicx}
\usepackage{faktor}
\usepackage{mathtools}
\usepackage{mathalfa}
\usepackage{units}
\usepackage{bbm}
\usepackage[inline]{enumitem} 
\usepackage{hyperref}
\usepackage{xfrac}
\usepackage{xcolor}
\usepackage{multicol}
\usepackage[normalem]{ulem}

\usepackage{array}
\usepackage{booktabs}
\usepackage{xcolor}
\usepackage{multirow}

\definecolor{headerblue}{RGB}{0, 102, 204}
\definecolor{rowgray}{RGB}{240, 240, 240}

\usepackage{mdframed}

\usepackage{tikz}
\usetikzlibrary{arrows}
\usetikzlibrary{arrows.meta} 
\usetikzlibrary{calc}
\usetikzlibrary{positioning} 

\PassOptionsToPackage{unicode}{hyperref}
\PassOptionsToPackage{naturalnames}{hyperref}

\newcommand{\defeq}{\vcentcolon=}
\newcommand{\R}{\mathbb{R}}

\newcommand{\Z}{\mathbb{Z}}
\renewcommand{\P}{\mathbb{P}}

\newcommand{\PP}{\mathcal{P}}

\newcommand{\MM}{\mathcal{M}}

\renewcommand{\SS}{\mathcal{S}}

\newcommand{\N}{\mathbb{N}}
\newcommand{\F}{\mathcal{F}}

\newcommand{\G}{\mathcal{G}}
\newcommand{\A}{\mathcal{A}}

\newcommand{\M}{\mathcal{M}}
\newcommand{\E}{\mathbb{E}}

\newcommand{\eps}{\varepsilon}

\newcommand{\dado}{ \, \left.\right| \, }

\newcommand{\nc}{\newcommand}

\nc{\I}{{\bf 1}}

\nc{\BP}{\mathbb{P}}
\nc{\BE}{\mathbb{E}}
\nc{\BQ}{\mathbb{Q}}
\nc{\BX}{\mathbb{X}}
\numberwithin{equation}{section}

\DeclareMathOperator*{\argmin}{arg\,min}
\DeclareMathOperator*{\argmax}{arg\,max}

\DeclareMathOperator*{\esssup}{ess\,sup}

\DeclareMathOperator*{\essinf}{ess\,inf}

\definecolor{ao(english)}{rgb}{0.0, 0.5, 0.0}

\def\matt#1{\textcolor{blue}{{[\bf Matt:} #1{\bf ]}}}

\def\vitto#1{\textcolor{ao(english)}{{[\bf Vitto:} #1{\bf ]}}}

\def\sf#1{\textcolor{magenta}{{[\bf Sergey:} #1{\bf ]}}}

\mdfdefinestyle{myenvs}{%
  hidealllines=true,%
  leftmargin=0pt,
  rightmargin=0pt,
  innerleftmargin=10pt,
  innerrightmargin=10pt,
  skipabove=\topsep,
    skipbelow=\topsep,
    nobreak=true, 
    splittopskip=2\topskip, 
    splitbottomskip=2\topskip,
    backgroundcolor=white, 
  roundcorner=5pt, 
  innertopmargin=6pt, 
  innerbottommargin=6pt, 
  linewidth=0pt 
}

\newmdtheoremenv[style=myenvs]{theorem}{Theorem}[section]
\newmdtheoremenv[style=myenvs]{thm}{Theorem}[section]
\newmdtheoremenv[style=myenvs]{definition}{Definition}[section]
\newmdtheoremenv[style=myenvs]{defn}{Definition}[section]
\newmdtheoremenv[style=myenvs]{example}{Example}[section]
\newmdtheoremenv[style=myenvs]{exmp}{Example}[section]
\newmdtheoremenv[style=myenvs]{proposition}{Proposition}[section]
\newmdtheoremenv[style=myenvs]{prop}{Proposition}[section]
\newmdtheoremenv[style=myenvs]{assumption}{Assumption}[section]
\newmdtheoremenv[style=myenvs]{ass}{Assumption}[section]
\newmdtheoremenv[style=myenvs]{notation}{Notation}[section]
\newmdtheoremenv[style=myenvs]{nota}{Notation}[section]
\newmdtheoremenv[style=myenvs]{remark}{Remark}[section]
\newmdtheoremenv[style=myenvs]{rmk}{Remark}[section]
\newmdtheoremenv[style=myenvs]{corollary}{Corollary}[section]
\newmdtheoremenv[style=myenvs]{cor}{Corollary}[section]
\newmdtheoremenv[style=myenvs]{lemma}{Lemma}[section]
\newmdtheoremenv[style=myenvs]{lema}{Lemma}[section]
\newmdtheoremenv[style=myenvs]{problem}{Problem}[section]
\newmdtheoremenv[style=myenvs]{hypothesis}{Hypothesis}[section]
\newmdtheoremenv[style=myenvs]{hyp}{Hypothesis}[section]
\newmdtheoremenv[style=myenvs]{conjecture}{Conjecture}[section]
\newmdtheoremenv[style=myenvs]{conj}{Conjecture}[section]
\newmdtheoremenv[style=myenvs]{model}{Model}[]  

\usepackage[doi=false,url=false,sorting=nyt,giveninits=true,maxbibnames=4]{biblatex}
\bibliography{main}

\title{On the instability of local learning algorithms:\\
Q-learning can fail in infinite state spaces\footnote{Corresponding author}}
	\author{U. Ayesta$^{a,b}$, S. Foss$^{d,c}$, M. Jonckheere $^{b,a}$ and V. Puricelli $^{a,b,*}$ \vspace{4pt}\\
		$^a$ CNRS, IRIT, 2 rue C. Camichel, 31071 Toulouse, France.\\
		$^b$ CNRS, LAAS, 7 av. du colonel Roche,
31031 Toulouse, France.\\
		$^c$ Maxwell Institute, Heriot-Watt University, EH14 4AS, Edinburgh, United Kingdom.\\
        $^d$ School of MACS, Heriot-Watt University, EH14 4AS, Edinburgh, United Kingdom. 
	}

\begin{document}

\maketitle

\begin{abstract}
  We investigate the challenges of applying model-free reinforcement learning algorithms, like online Q-learning, to infinite state space Markov Decision Processes (MDPs). 
  We first introduce the notion of Local Learning Processes (LLPs), where agents make decisions based solely on local information, and we show that Q-learning can be seen as a specific instance of an LLP. Using renewal techniques, we analyze LLPs and demonstrate their instability under certain drift and initial conditions, revealing fundamental limitations in infinite state spaces. 
  In particular, we show that while asymptotically optimal in finite settings, Q-learning can face instability and strict sub-optimality in infinite spaces.
  Our findings are illustrated through queueing system examples drawn from load balancing and server  allocation. The study underscores the need for new theoretical frameworks and suggests future research into nonlocal Q-learning variants.
\end{abstract}


\section{Introduction}

In the realm of Reinforcement Learning (RL) and Markov Decision Processes (MDPs), simple algorithms such as Q-learning or SARSA stand out as easy-to-implement and successful approaches for solving complex stochastic decision-making problems.
In particular, Q-learning has garnered significant attention as an algorithm with the ability to learn optimal policies without requiring modeling the underlying system dynamics. It is based on bootstrap estimates of the optimal action-value function which balance exploration and exploitation. A critical theoretical foundation for Q-learning is its convergence to the optimal policy—a result rigorously established under the assumption of finiteness of the state space in
\cite{Watkins92} and \cite{Tsitsiklis94}.

In many real-world problems, however, such as resource allocation in communication networks or control of physical systems, natural models have a countable state space. This is the starting point of the present article, where we discuss value-based online learning algorithms on discrete but infinite state spaces.

In the context of Markov Decision Processes (MDPs) (where the transitions dynamics are supposed to be known), several recent studies have explored the challenges associated with infinite state spaces. For instance, \cite{grosof2024} addressed the convergence of the Natural Policy Gradient (NPG) algorithm on some queueing MDPs in infinite state spaces, providing insights into how finite-state approaches can be extended.
 These generalizations typically rely on the structural properties of the state space and on the continuity of the reward and transition functions, which allow for the extension of convergence proofs from finite to infinite settings, both for discounted and undiscounted settings.  Also their method works under the additional assumption on the initial policy being stable.
In the same vein, \cite{Dai-gluzman} considers a queueing system with a parameterized policy and constructs a practical algorithm that, initialized with a stable policy, iteratively improves the average reward.



One can then naturally ask whether similar results hold when dealing with model-free approaches.  These methods, which do not assume knowledge of the underlying MDP dynamics, face significant and different challenges from those arising in MDPs with infinite state spaces. 
The lack of a prior knowledge of state transitions and rewards can lead to instability and to convergence issues that are not present in model-based approaches. This discrepancy highlights the need for new theoretical frameworks and algorithms specifically designed to address infinite state spaces in reinforcement learning.
It is remarkable that the literature is much scarcer in this context. In \cite{shahetal2020}, 
the authors look at the generative case (where one can sample transitions from any pair of state/action). 
They show that an algorithm based on tabular natural policy gradient with softmax parametrization, using the notion of oracle based Monte-Carlo estimates of the Q-function (and restart), is able to find a stable optimal policy if there exists one. 
To the best of our knowledge, there is no results on online RL (i.e., without a generative simulator) in this context.

 Our goal is hence three-fold:
1. Define a class of algorithms, including online Q-learning, having a crucial local property.
2. Show that, under an additional assumption of space invariance, those algorithms can be studied using a framework that extends the classical theory of Markov chains.
3. Demonstrate that, under certain parameter conditions, these algorithms can exhibit instability, implying the inability to learn the optimal policy (for any local algorithm).

To pursue this program, we introduce a key concept of local algorithm coined Local Learning Process (LLP). An LLP is a discrete-time stochastic process used to model the behavior of an agent navigating a (possibly infinite) state space, where the agent sequentially decides which action to take at each state without prior knowledge of the transition probabilities. 
The locality refers to the agent's decision-making and learning being based only on the information available at its current state, this information is updated at each state using solely immediate transition data, without considering more global information involving all possible states of the environment. 


Building on this definition, our main tool for analyzing the stochastic behavior of the state process will be renewal techniques. The primary challenge is that, due to the dependency on learning features, this process itself is not inherently Markovian. However, we are able instead to establish nonstandard regenerative structure where the regenerative cycles may depend on the whole future. Renewal theory is  key to overcome
the difficulty, allowing us to study independent cycles of the process. We consider that one of our key
contributions is demonstrating how renewal techniques may be applied in the context of LLPs.


There is a number of other intricate stochastic processes where similar ideas were used for the construction of independent and identically distributed (i.i.d.) regenerative cycles that can, in general, depend on  the whole future and past.
These processes often involve complex dependencies and require innovative approaches to analyze their behavior over time.
A key area illustrating this challenge is the study of Random Walks in Random Environment (RWRE). In these models, the fundamental parameters, such as transition probabilities, are not fixed but are randomly sampled from a distribution that may itself evolve, leading to rich and intricate behaviors \cite{sznitman-zerner}.
A natural extension of this framework is the Cookie Random Walk, where the environment is locally modified by ``cookies'' placed at each state. These cookies act as a dynamic context, altering transition probabilities as the walker encounters them. This setup provides a versatile model for systems 
 where local conditions or external factors influence the behavior of the system, adding another layer of complexity to the analysis of stochastic processes.

The analytical techniques developed for these walks, for which we refer to \cite{zerner06, ramirez07, menshikov12}, share common ground with other domains. Notably, the concept of regenerative cycles arises in diverse stochastic models such as directed random graphs \cite{foss2003extended, foss2024last}, processes with long memory \cite{comets2002processes}, and interacting particle systems like contact processes \cite{kuczek1989central, mountford2000extension} and infection processes \cite{foss2013stochastic}. Our work builds upon ideas from this broad literature, with particular inspiration drawn from \cite{foss2024last, foss2013stochastic}.

We introduce here a model that captures fundamental resource allocation problems like load balancing and server allocation, which we develop in detail. Many other similar examples could be considered. This model defines a family of MDPs for which we can develop the renewal techniques previously mentioned. These techniques ultimately establish the instability (and hence nonoptimality) of LLPs under certain initial conditions. In this model, states are pairs of nonnegative integers, representing the amounts of jobs in a pair of queues. Two possible actions are available at each state, for example in load balancing, an action specifies which server receives an incoming job, while in server allocation, it determines which server to activate. Each action leads to a different transition dynamics which are assumed to have a homogeneous structure when neither queue is empty ---crucial for renewal techniques to be applied--- with at least one action choice leading to stable behavior. 

The main result of this work, Theorem~\ref{thm:MainResult}, demonstrates that, under certain drift and initial conditions, any LLP algorithm exhibits transient behavior. 
This transience implies that the agent policy will not stabilize, leading to suboptimal performance. This result is significant as it highlights a fundamental limitation of Q-learning and of similar algorithms in infinite state spaces, where traditional convergence guarantees may not hold. 
Interestingly, we present an example illustrating that while both ``red'' and ``green'' policies individually lead to a stable process, an adequate level of exploration results in a convex combination of these policies that ultimately becomes unstable. This phenomenon is connected to the Parrondo's paradox, which states, in particular, that a convex combination of stable Markov kernels can yield an unstable process, see e.g. \cite{HA1999} or \cite{KKA2006}. 

Though we focused on negative results here, our techniques could also lead to identifying sufficient conditions for LLPs to be stable. However, in a model free context, these conditions, depending in a intricate way on the underlying dynamics and exploration levels are not exploitable and we leave for future work positive results concerning nonlocal versions of Q-learning.

The paper is structured as follows: Section~\ref{sec:mainsetting} introduces the formal definition of LLPs and describes Model \ref{Model} in detail. Section 3 is devoted to present a relaxation of Model~\ref{Model} and establishes a corresponding limit theorem for LLPs within this new framework. Section~\ref{sec:mainresult} presents the main theoretical result Theorem~\ref{thm:MainResult} and a  sketch of its proof. Section~\ref{sec:numerics} illustrates numerically our findings and explores some generalizations. Section~\ref{sec:conclusion} discusses the implications of these findings and potential directions for future research. 
Finally, the Appendix provides the detailed mathematical constructions for the instability analysis and contains the proofs of all results presented in this work. 
It is worth noting that although the jumps are bounded in our Models 1 and 2, the techniques developed in the Appendix can be applied, with minor modifications, to models where the jump distributions have unbounded support but finite exponential moments.

An extended abstract of this article appeared in \cite{ifip-short-anon}.

\section{Main setting}\label{sec:mainsetting}

\subsection{Local Learning Processes} Consider an infinite state space $\SS   \subset \Z^d$, an action space $\A = \{r,g\}$ and transition probability function $p: \SS \times \SS \times \A \to [0,1]$, where $p( y \ | \ x , a )$ stands for the probability of going from $x $ to $y$ by applying action $a$. The main notation used in the paper work are summarized in Table~\ref{glossary}.


Our objective is to model an agent that has no knowledge of $p( y \dado x , a)$ and must decide sequentially which action to take at each state. 
We then investigate whether the system remains stable according to Definition~\ref{defn:unstable} below.

We allow the agent to ``store'' information at each state, which will be used to decide which action to take upon subsequent visits to that state. After each visit, the agent has the ability to update the stored information. The information that the agent stores at each state is what we refer to as the ``environment'' and it is formally defined as a function $\beta: \SS \times \A \to \R$.

We now introduce a key concept in our work, the notion of Local Learning Process (LLP). We first describe it as follows. Starting from $X_0 \in \SS$ and initial \emph{environment} $\beta_0: \SS \times \A \to \R$, a \emph{decision rule} $\varphi$ determines the law of the action to be taken as a function of the environment at the current state: $A_0 \sim \varphi (\ \cdot \dado \beta_0(X_0, \cdot))$. We assume that $\varphi (\ a \dado \beta_0( x, \cdot)) > 0$ for any action $a$ and state $x$, meaning that there is some positive probability of applying any action at the first visit to any state. Then action $A_0$ is applied at state $X_0$ and the process moves accordingly to $X_1 \sim p(\ \cdot \dado X_0, A_0)$. Finally, an \emph{update rule} changes the value of the environment at $(X_0,A_0)$, defining a new environment $\beta_1$ and the process goes on in the same fashion for all times $n \geq 1$. 

We now provide a formal definition of LLP that captures processes arising from popular learning algorithms, such as Q-learning as we will see later in Definition \ref{defn:Qlearningpolicy}.

\begin{defn}[Local Learning Process] \label{defn:learningprocess} ~\\
    A Local Learning Process (LLP) is a discrete time stochastic process $(X_n,A_n,\beta_n)_{n\geq 0}$ defined by an initial state $X_0 \in \SS$, an initial environment $\beta_0: \SS \times \A \to \R$, an update rule $\psi$ and a decision rule $\varphi$ that are intertwined by the following equations
\begin{align*}
    &1. \ \  A_n \sim \varphi(\ \cdot \dado \beta_n(X_n, \cdot) ) \\
    &2. \ \ X_{n+1} \sim p(\  \cdot \dado X_n , A_n ), \\
    &3. \ \ \beta_{n+1}(X_n,A_n) = \psi (\beta_n(X_{n},\cdot),\beta_n(X_{n+1}, \cdot),X_{n+1} - X_n , A_n),
\end{align*}
where we let $\beta_{n+1}(x,a) =
        \beta_n (x,a) \text{ if } (x,a) \neq (X_n,A_n).$
We assume further that an LLP applies with a positive probability any action $a$ at the first visit to any state:
\begin{align}
    \varphi ( \ a \dado \beta_0 ( x , \cdot ) ) > 0 \text{ for all } x \in \SS, \text{ and } a\in\A.  \label{eq:LLP red}
\end{align}
We finally assume that the initial environment is state invariant:
\begin{align}\label{homogeneous environment}
    \beta_0 (x, \cdot ) = \beta_0 (y, \cdot ) \ \text{ for all } x,y \in \SS.
\end{align}
\end{defn}

    Any LLP defines a probability distribution $\P$ on state-action paths such that for all $x_0,x_1,\dots,x_{N+1} \in \SS$ and $a_0,\dots,a_{N} \in \A$ 
    \begin{align}
        & \P(A_0 = a_0, X_1 = x_1, \dots, A_{N} = a_{N}, X_{N+1} = x_{N+1} \ | \ X_0 = x_0 ) \notag \\
        & \quad=   \varphi (\ a_0 \dado \beta_0 (x_0, \cdot ) ) p (x_1 \dado x_0, a_0 )   \dots  \varphi (\ a_{N} \dado \beta_{N} (x_{N}, \cdot ) ) p (x_{N+1} \dado  x_{N}, a_{N} )  \label{eq:law at time n} 
        \end{align}
     This is justified because $\beta_i$, the environment at time $i \geq 1$, can be computed recursively from the initial environment $\beta_0$, and the state-action trajectory $x_0,a_0,x_1,a_1,\dots,x_i,a_i,x_{i+1}$.

   Let us move now to introducing the concept of \emph{stability} which is a central property for learning tasks on infinite state spaces. The main result of this work Theorem \ref{thm:MainResult} establishes the \emph{transience} of LLPs in a concrete family of examples.

 \begin{defn}\label{defn:unstable}~\\
    \begin{itemize}
        \item We say that $(X_n)_{n \geq 0}$ is irreducible if, for any bounded, nonempty set $U\subset \SS$ and any $x_0 \notin U$, 
        \begin{equation}
    \P( \inf \{ n\geq 1: X_n \in U \} < + \infty \ | \ X_0 = x_0) > 0. 
\end{equation}
        \item We say that $(X_n)_{n \geq 0}$ is \emph{transient} or \emph{unstable} if it is irreducible and, for any bounded set $U \subset \SS$, there is $x_0 \notin U$  such that 
    \begin{equation}\label{eq:transience}
        \P ( \  \inf \{ n \geq 1 : X_n \in U \} = +\infty \ \dado X_0 = x_0 ) > 0 .
    \end{equation}
    \item We say that $(X_n)_{n\geq 0}$ is \emph{positive recurrent} or \emph{stable} if it is irreducible and, for any bounded set $U \subset \SS$ and any $x_0 \notin U$,
    \begin{equation}
        \E ( \  \inf \{ n \geq 1 : X_n \in U \} ) \ \dado X_0 = x_0 ) < + \infty .
    \end{equation}
    \end{itemize}
\end{defn} 


Observe that the Definition \ref{defn:unstable} coincides with the standard definitions of irreducibility, transience and positive recurrence for Markov chains. Note that for irreducible countable aperiodic Markov chains {\it stability} usually means {\it positive recurrence}, so transience is a strong form of instability.
In the sequel, all the MDPs we shall consider (see Model \ref{Model} below) have (sufficiently) irreducible dynamics so we do not elaborate further on unichain conditions here and the previous definition will be sufficient for our needs.


\subsection{Model \ref{Model}: queueing networks}\label{sec:loadbalancing}

We now provide a detailed description of  the model of $(\SS,\A,p)$, that is required for the formulation of  Theorem \ref{thm:MainResult}. 

We need a number of further notation for the precise description of Model \ref{Model}. Let $\PP$ be the family of probability measures on the 4-states space $\{e_1,-e_1,e_2,-e_2 \}$, where $e_1=(1,0)$ and $e_2=(0,1)$ are the canonical vectors in ${\mathbb R}^2$. For any $\mu \in \PP$, denote its {\it drift} as $
    d(\mu) = ( \mu(e_1) - \mu(-e_1) , \mu(e_2) - \mu(-e_2) ) \in \R^2$. The positive orthant is $\R^2_+ = \{ (x_1,x_2) \in \R^2 : x_1,x_2 \geq 0 \}$ and its restriction to the integer grid ${\mathbb Z}_{+}^2 = \{(x_1,x_2) \in \Z^2 : x_1,x_2\geq 0 \}$.

\begin{model}\label{Model}
     Assume that, in the triple $(\SS,\A,p)$, we let  $\SS = {\mathbb Z}_{+}^2$, $\A = \{r,g\}$ and the transitions probabilities $p( \ y \dado x , a )$ defined as follows: for  $x = (x_1,x_2), a = r,g$ there are given   $\mu_a,\mu_a ', \mu_a '' \in \PP$ such that 
     \begin{align} 
         &p(x \pm e_i \dado x, a) = \begin{cases}
             \mu_a (\pm e_i) & \text{ if } x_1 > 0 \text{ and } x_2 > 0, \\
             \mu ' _a (\pm e_i) & \text{ if } x_1 > 0 \text{ and } x_2 = 0, \\
             \mu '' _a (\pm e_i) & \text{ if } x_1 = 0 \text{ and } x_2 > 0. \label{eq: transitionsmodel1}
         \end{cases} 
     \end{align}
    For the state $x = (0,0)$ we simply assume that $p( e_1 \dado (0,0) ,a ) + p( e_2 \dado (0,0) ,a ) = 1$ for $a=r,g$.
     We assume the following conditions upon the drifts. 
     
     There is $v \in \R^2_+$ such that     \begin{equation}\label{parameter2}
         \langle w , v \rangle < 0 \text{ for } w = d_g,d_g ', d_g '',
     \end{equation}
     There is a unit vector $l  \in \R^2$ such that 
\begin{equation}
\varrho \defeq \min \{ \langle d_r , l \rangle, \langle d_g , l \rangle \} > 0 \label{parameter4}
\end{equation}
Finally, assume that there exists $0< \alpha_0 < \alpha_1 \leq 1$ such that $\delta d_r + (1-\delta) d_g \in \R^2_+$ for all $\delta \in  (\alpha_0,\alpha_1)$, and satisfying
\begin{equation}\label{parameter5}
    1-\alpha_1 < \varrho - \alpha_0.
\end{equation}
\end{model}

As mentioned in the Introduction, Model \ref{Model} captures common features of (simple) queueing problems (see Examples 1 and 2 below). 
Particularly \eqref{parameter2} implies the following stability result for an agent always applying ``green'' action. We leave the proof for Appendix \ref{appendix:loadbalancing}.
\begin{lema}\label{lemma:lyapunovgreen}
    Let $(\SS,\A,p)$ be as in Model \ref{Model} and $(X_n)_{n\geq 0}$ defined by $X_{n+1} \sim p( \cdot \ | \ X_n , g )$ for all $n \geq 0$. Then $(X_n)_{n\geq 0}$ is positive recurrent, and if $v \in \R^2_+$ satisfies \eqref{parameter2}, then  
    \begin{equation}
        \E( \langle X_\infty , v \rangle ) < + \infty.
    \end{equation}
\end{lema}

 Next, condition \eqref{parameter4} states that there is a common direction for both drifts $d(\mu_r)$ and $d(\mu_g)$; in other words these drifts are not in opposite directions. This is crucial for the arguments we develop for our main result Theorem \ref{thm:MainResult}, particularly in Appendix \ref{appendix:mathematicalframework}. Finally the existence of the quantities $\alpha_0$ and $\alpha_1$ imply that there are common directions in the positive orthant $\R^2_+$, and \eqref{parameter5} appears as a technical condition for the proof of Theorem \ref{thm:MainResult}. 
 Figure~\ref{fig:cond model 1} illustrates Conditions \eqref{parameter4} and \eqref{parameter5} for the models considered in the Examples 1 and 2 described below.

\begin{figure}[ht]
\centering
\resizebox{0.7\textwidth}{!}{
\begin{tikzpicture}
    \coordinate (dr_end) at (1.7,1.5);
    \coordinate (dg_end) at (2,-0.8);
    
    \clip (-1.1,-1.5) rectangle (3.1,3.1);
    
    \coordinate (dr_perp) at (-1.5,1.7);
    
    \coordinate (dg_perp) at (0.8,2);
    
    \begin{scope}
        \path[clip] ($-5*(dr_perp)$) -- ($5*(dr_perp)$) -- ($5*(dr_perp) + 8*(dr_end)$) -- ($-5*(dr_perp) + 8*(dr_end)$) -- cycle;
        \fill[gray!20, opacity=0.3] (-8,-8) rectangle (8,8);
    \end{scope}
    
    \begin{scope}
        \path[clip] ($-5*(dg_perp)$) -- ($5*(dg_perp)$) -- ($5*(dg_perp) + 8*(dg_end)$) -- ($-5*(dg_perp) + 8*(dg_end)$) -- cycle;
        \fill[gray!20, opacity=0.4] (-8,-8) rectangle (8,8);
    \end{scope}
    
    \begin{scope}
        \clip ($-5*(dr_perp)$) -- ($5*(dr_perp)$) -- ($5*(dr_perp) + 8*(dr_end)$) -- ($-5*(dr_perp) + 8*(dr_end)$) -- cycle;
        \path[clip] ($-5*(dg_perp)$) -- ($5*(dg_perp)$) -- ($5*(dg_perp) + 8*(dg_end)$) -- ($-5*(dg_perp) + 8*(dg_end)$) -- cycle;
        \fill[gray!60, opacity=0.8] (-8,-8) rectangle (8,8);
    \end{scope}
    
    \draw[dotted, thick, gray] ($-4*(dr_perp)$) -- ($4*(dr_perp)$);
    \draw[dotted, thick, gray] ($-4*(dg_perp)$) -- ($4*(dg_perp)$);
    
    \draw[dotted, thick, black] (dr_end) -- (dg_end);
    
    \draw[black,  thick, -{Latex[length=5pt]}] (0,0) -- (dr_end) node[above left, font=\scriptsize] {$d(\mu_r)$};
    
    \draw[black, thick, -{Latex[length=5pt]}] (0,0) -- (dg_end) node[below , font=\scriptsize] {$d(\mu_g)$};
    
    \draw[black, thin, -] (-1,0) -- (4,0) node[below, font=\scriptsize] {};
    \draw[black, thin, -] (0,-2) -- (0,4) node[right, font=\scriptsize] {};
    
\end{tikzpicture} \hspace{2cm} \begin{tikzpicture}
    \coordinate (dr_end) at (-0.7,1.5);
    \coordinate (dg_end) at (2,-0.8);
    
    \clip (-1.1,-1.5) rectangle (3.1,3.1);
    
    \coordinate (dr_perp) at (-1.5,-0.7);
    
    \coordinate (dg_perp) at (0.8,2);
    
    \begin{scope}
        \path[clip] ($-5*(dr_perp)$) -- ($5*(dr_perp)$) -- ($5*(dr_perp) + 8*(dr_end)$) -- ($-5*(dr_perp) + 8*(dr_end)$) -- cycle;
        \fill[gray!20, opacity=0.3] (-8,-8) rectangle (8,8);
    \end{scope}
    
    \begin{scope}
        \path[clip] ($-5*(dg_perp)$) -- ($5*(dg_perp)$) -- ($5*(dg_perp) + 8*(dg_end)$) -- ($-5*(dg_perp) + 8*(dg_end)$) -- cycle;
        \fill[gray!20, opacity=0.4] (-8,-8) rectangle (8,8);
    \end{scope}
    
    \begin{scope}
        \clip ($-5*(dr_perp)$) -- ($5*(dr_perp)$) -- ($5*(dr_perp) + 8*(dr_end)$) -- ($-5*(dr_perp) + 8*(dr_end)$) -- cycle;
        \path[clip] ($-5*(dg_perp)$) -- ($5*(dg_perp)$) -- ($5*(dg_perp) + 8*(dg_end)$) -- ($-5*(dg_perp) + 8*(dg_end)$) -- cycle;
        \fill[gray!60, opacity=0.8] (-8,-8) rectangle (8,8);
    \end{scope}
    
    \draw[dotted, thick, gray] ($-4*(dr_perp)$) -- ($4*(dr_perp)$);
    \draw[dotted, thick, gray] ($-4*(dg_perp)$) -- ($4*(dg_perp)$);
    
    \draw[dotted, thick, black] (dr_end) -- (dg_end);
    
    \draw[black,  thick, -{Latex[length=6pt]}] (0,0) -- (dr_end) node[above , font=\scriptsize] {$d(\mu_r)$};
    
    \draw[black, thick, -{Latex[length=6pt]}] (0,0) -- (dg_end) node[below , font=\scriptsize] {$d(\mu_g)$};
    
    \draw[black, thin, -] (-1,0) -- (4,0) node[below, font=\scriptsize] {};
    \draw[black, thin, -] (0,-2) -- (0,4) node[right, font=\scriptsize] {};
    
\end{tikzpicture}
}
    \caption{Representation of Conditions \eqref{parameter4} and \eqref{parameter5} of Model \ref{Model} in two examples described below: Example 1 (left) and Example 2 (right). The darker area represents the set of vectors that have a positive component with both drifts $d(\mu_r)$ and $d(\mu_g)$.}
    \label{fig:cond model 1} 
    
\end{figure}
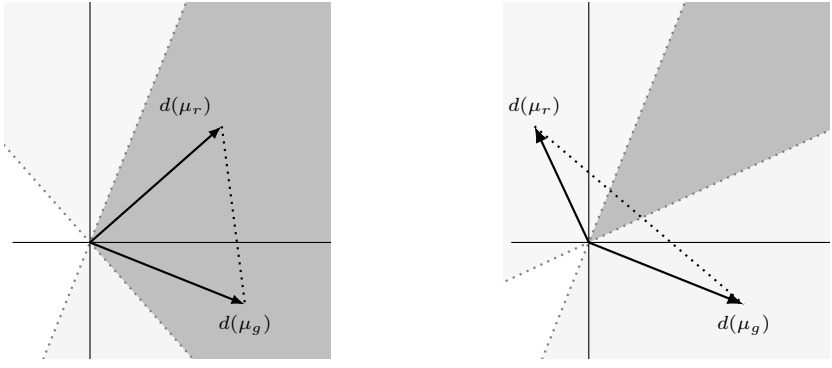

We now give two practical examples of Model 1 representing two types of queuing networks. In these examples, each state corresponds to a point $x = ( x_1, x_2) \in \Z^2_+$ indicating the lengths of both queues (how many jobs are in each queue). Formally the rates and parameters considered in the two examples below define continuous-time Markov chains on $\Z^2_{+}$. However, it is well-known that given a continuous-time Markov chain on a discrete state space, one can construct a corresponding discrete-time Markov chain by observing the system at the random times when state transitions occur. This yields what is known as the corresponding embedded Markov chain. The transition probabilities of the embedded Markov chain capture the probability of jumping from one state to another, conditioned on the fact that a transition occurs. This construction defines the transition probabilities $p( y \ | \ x, a)$ considered herein. Explicit formulas are derived for Examples 1 and 2 (given below) in Appendix \ref{appendix:loadbalancing}.

\textbf{Example 1: load balancing.} Consider jobs arriving to a two-queue system with rate $\lambda$. Upon arrival, a job is sent by the dispatcher to Queue 1 with probability $p\in [0,1]$ and to Queue 2 with probability $1-p$. When both queues are busy (nonempty), jobs leave Queue $i$ with rate $\mu_i$ for $i = 1,2$. When one (no matter which one) queue is empty, jobs leave the busy queue with rate $\widetilde{\mu}$. Further, the dispatcher allows only two possibilities for the parameter $p$, namely $p_r,p_g \in [0,1]$, for assigning the job to Queue 1. So the agent has only two available actions at each state, namely action ``red'' and action ``green''. A proof of the following Lemma can be found in Appendix \ref{appendix:loadbalancing}

\begin{lema}\label{lemma:loadbalancing.}
Example 1 satisfies the conditions of Model \ref{Model} if:
\begin{align}
    &\begin{cases}  
        \mu_1 < \lambda p_r \\     
        \mu_2 < \lambda (1-p_r) .  
    \end{cases} 
    \label{eq:loadbalancing1}\\
   & \begin{cases}   
        \mu_1 < \lambda p_g \\  
        \mu_2 > \lambda (1-p_g).
    \end{cases}
    \label{eq:loadbalancing2}\\
    & \widetilde{\mu} > \lambda \left(\frac{(\mu_1+\mu_2)p_g-\mu_1}{\mu_2 - \lambda (1-p_g)} \right).
    \label{eq:loadbalancing4}\\
    &\mu_2 < \frac{(\lambda p_r - \mu_1)(p_g-p_r)}{\lambda+\mu_1+\mu_2} + \lambda (1-p_g). 
    \label{eq:loadbalancing3}
\end{align}
    
Also the Markov chain $(X_n ^a)_{n\geq 0}$ defined by an agent always applying action $a$ is stable when $a = g$ and transient when $a = r$.
 
\end{lema}

 Let us mention that \eqref{eq:loadbalancing1} is linked to the instability of the ``red'' action, also \eqref{eq:loadbalancing1} in combination with \eqref{eq:loadbalancing2} gives \eqref{parameter4}. Then \eqref{eq:loadbalancing2} and \eqref{eq:loadbalancing4} are related to \eqref{parameter2} while \eqref{eq:loadbalancing3} corresponds to \eqref{parameter5}.

 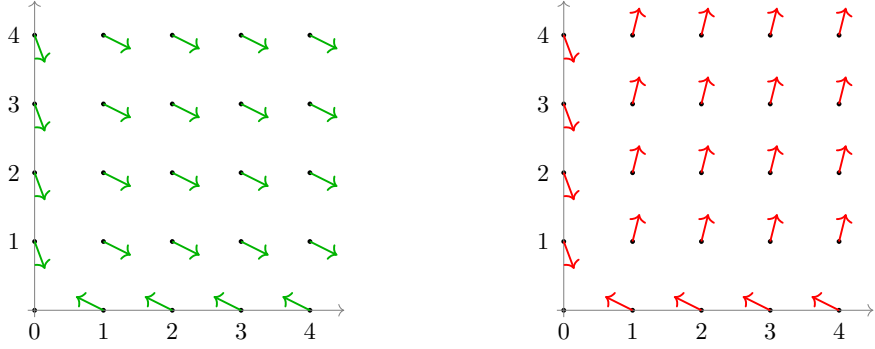
\begin{figure}[ht]
 \centering
        \resizebox{0.3\textwidth}{!}{ 
            \begin{tikzpicture}
                \foreach \x in {0,1,2,3,4} {
                    \foreach \y in {0,1,2,3,4} {
                        \fill (\x, \y) circle (1pt); 
                    }
                }

                \foreach \x in {0,1,2,3,4} {
                    \node at (\x, -0.3) {\x}; 
                }
                \foreach \y in {1,2,3,4} {
                    \node at (-0.3, \y) {\y}; 
                }

                \draw[->,thin, gray] (-0.1, 0) -- (4.5, 0) ; 
                \draw[->,thin, gray] (0, -0.1) -- (0, 4.5) ; 

                \foreach \x in {1,2,3,4} {
                    \foreach \y in {1,2,3,4} {
                        \draw[->, thick, green!70!black!100] (\x,\y) -- (\x+0.4,\y-0.2); 
                    }
                }

                \foreach \x in {1,2,3,4} {
                    \draw[->, thick, green!70!black!100] (\x,0) -- (\x-0.4,0.2); 
                }
                
                \foreach \y in {1,2,3,4} {
                    \draw[->, thick, green!70!black!100] (0,\y) -- (0.15,\y-0.4); 
                }
            \end{tikzpicture} }  \hspace{2cm}     \resizebox{0.3\textwidth}{!}{ 
            \begin{tikzpicture}
                \foreach \x in {0,1,2,3,4} {
                    \foreach \y in {0,1,2,3,4} {
                        \fill (\x, \y) circle (1pt); 
                    }
                }

                \foreach \x in {0,1,2,3,4} {
                    \node at (\x, -0.3) {\x}; 
                }
                \foreach \y in {1,2,3,4} {
                    \node at (-0.3, \y) {\y}; 
                }

                \draw[->,thin, gray] (-0.1, 0) -- (4.5, 0) ; 
                \draw[->,thin, gray] (0, -0.1) -- (0, 4.5) ; 

                \foreach \x in {1,2,3,4} {
                    \foreach \y in {1,2,3,4} {
                        \draw[->, thick, red] (\x,\y) -- (\x+0.1,\y+0.4); 
                    }
                }

                \foreach \x in {1,2,3,4} {
                    \draw[->, thick, red] (\x,0) -- (\x-0.4,0.2); 
                }
                
                \foreach \y in {1,2,3,4} {
                    \draw[->, thick, red] (0,\y) -- (0.15,\y-0.4); 
                }
            \end{tikzpicture}
        }
        \caption{Map of the drifts for the ``green'' action \emph{(left)} and ``red'' action \emph{(right)} for Example 1 (load balancing).}
     \label{fig:red-greensystem loadbalancing}
\end{figure}

\

\textbf{Example 2: server allocation.} Consider two types of jobs arriving at a single server with the same rate $\lambda$. Jobs of type $i$ arrive at Queue $i$ for $i = 1, 2$. The server must select which queue to serve. When both queues are busy, the server chooses to serve a job from Queue 1 (Queue 2) by applying action $r$ (action $g$). The selected job then leaves the system with rate $\mu$. When one of the queues is empty, the server must serve the nonempty queue for both actions, and in this case, the served job departs the system with rate $\widetilde{\mu}$.

As in the load balancing example each state corresponds to a state $x = ( x_1, x_2) \in \Z^2_{+}$ indicating the length of both queues (how many jobs are in each queue). Also we work with the transitions $p ( \ y \ | \ x, a)$ of the associated Embedded Chain as described previously. A proof of the following Lemma can be found in Appendix \ref{appendix:loadbalancing}.
\begin{lema}\label{lemma:serverallocation.}
Example 2 satisfies the drift conditions of Model \ref{Model} if: \begin{align}
& \lambda < \mu < 2\lambda < \widetilde{\mu}
\label{eq:singleserver1b}
\\
& \widetilde{\mu} > \frac{\mu \lambda}{\mu - \lambda}
\label{eq:singleserver2b}
\\ & \frac{\mu - \lambda}{\mu} < \frac{1}{\sqrt{2}} \frac{2\lambda - \mu}{2\lambda + \mu}
\label{eq:singleserver3b}
\end{align}

Also the Markov chain $(X_n ^a)_{n\geq 0}$ defined by an agent consistently applying either action $r$ or $g$ is stable in both cases.
\end{lema}

We mention that \eqref{eq:singleserver1b} and \eqref{eq:singleserver2b} imply both \eqref{parameter2} and \eqref{parameter4}. They also imply the stability of both actions, ``red'' and ``green''. Finally \eqref{eq:singleserver3b} is related to \eqref{parameter5}.

 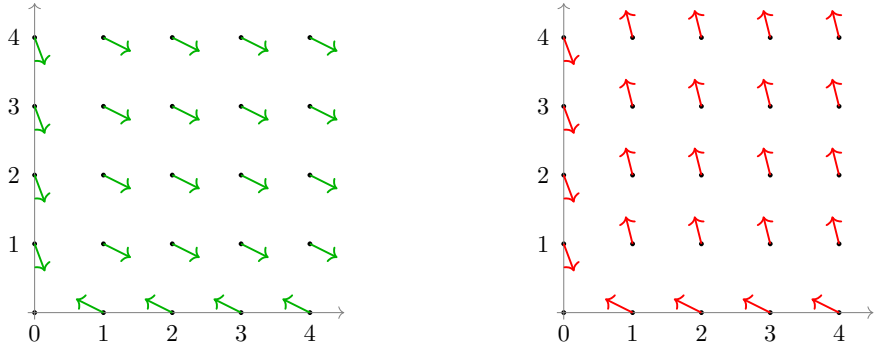
\begin{figure}[ht]
 \centering
        \resizebox{0.3\textwidth}{!}{ 
            \begin{tikzpicture}
                \foreach \x in {0,1,2,3,4} {
                    \foreach \y in {0,1,2,3,4} {
                        \fill (\x, \y) circle (1pt); 
                    }
                }

                \foreach \x in {0,1,2,3,4} {
                    \node at (\x, -0.3) {\x}; 
                }
                \foreach \y in {1,2,3,4} {
                    \node at (-0.3, \y) {\y}; 
                }

                \draw[->,thin, gray] (-0.1, 0) -- (4.5, 0) ; 
                \draw[->,thin, gray] (0, -0.1) -- (0, 4.5) ; 

                \foreach \x in {1,2,3,4} {
                    \foreach \y in {1,2,3,4} {
                        \draw[->, thick, green!70!black!100] (\x,\y) -- (\x+0.4,\y-0.2); 
                    }
                }

                \foreach \x in {1,2,3,4} {
                    \draw[->, thick, green!70!black!100] (\x,0) -- (\x-0.4,0.2); 
                }
                
                \foreach \y in {1,2,3,4} {
                    \draw[->, thick, green!70!black!100] (0,\y) -- (0.15,\y-0.4); 
                }
            \end{tikzpicture} }  \hspace{2cm}     \resizebox{0.3\textwidth}{!}{ 
            \begin{tikzpicture}
                \foreach \x in {0,1,2,3,4} {
                    \foreach \y in {0,1,2,3,4} {
                        \fill (\x, \y) circle (1pt); 
                    }
                }

                \foreach \x in {0,1,2,3,4} {
                    \node at (\x, -0.3) {\x}; 
                }
                \foreach \y in {1,2,3,4} {
                    \node at (-0.3, \y) {\y}; 
                }

                \draw[->,thin, gray] (-0.1, 0) -- (4.5, 0) ; 
                \draw[->,thin, gray] (0, -0.1) -- (0, 4.5) ; 

                \foreach \x in {1,2,3,4} {
                    \foreach \y in {1,2,3,4} {
                        \draw[->, thick, red] (\x,\y) -- (\x-0.1,\y+0.4); 
                    }
                }

                \foreach \x in {1,2,3,4} {
                    \draw[->, thick, red] (\x,0) -- (\x-0.4,0.2); 
                }
                
                \foreach \y in {1,2,3,4} {
                    \draw[->, thick, red] (0,\y) -- (0.15,\y-0.4); 
                }
            \end{tikzpicture}
        }
        \caption{Map of the drifts for the ``green'' action \emph{(left)} and ``red'' action \emph{(right)} for Example 2 (server allocation).}
     \label{fig:red-greensystem serverallocation}
\end{figure}

\subsection{Online Q-learning}

We now illustrate the LLP framework with a concrete example: an online version of the Q-learning algorithm, one of the most popular algorithms of the Reinforcement Learning literature. For this we introduce the following necessary preliminaries.

In the discounted cost problem the goal for an agent is to minimize for some fixed $\gamma \in (0,1]$ the discounted cost $J_\gamma$ or the total cost $J = J_1$:
\begin{equation}\label{eq:discountedcost}
   J_\gamma (x) \defeq \limsup_n \E \left(  \sum_{i=0} ^ {n-1} \gamma^i c(X_i,A_i,X_{i+1}) \dado X_0 = x \right).
\end{equation} where $c(x,a,y)$ is the cost of moving to state $y$ by applying action $a$ at state $x$.

A policy is $\pi:\SS \to \PP(\A)$, where $\pi( a \ | \ x )$ is the probability of making action $a$ at state $x$ under policy $\pi$. The objective is to find an optimal policy $\pi^*$ such that
\begin{equation*}
    J_\gamma ^{\pi^*} ( x ) = J^*_\gamma (x) := \min_{\pi} J_\gamma ^{\pi} (x) 
\end{equation*}

It is known that there exists an optimal state-value function $Q^* : \SS \times \A \to \R$, which is known to be the unique solution of the so-called Bellman equation (see \cite{puterman}), that for $\gamma <1$ takes the form:
\begin{align*}
    &Q^* (s,a) = \min_{a' \in \A} \left\{  \sum_{y \in \SS } p(y \ | \ x , a ) \left( c(x,a,y)+ \gamma Q^* (y,a') \right)\right\}.
\end{align*}

The knowledge of $Q^*$ allows to define an optimal policy $\pi^*$  by considering $\A^*(x) = \argmin_{a\in \A} Q^* (x,a)$ and letting
\begin{equation*}
    \pi^* ( a \ | \ x ) = \begin{cases}
        \dfrac{1}{| \A^*(x)|} & \text{ if } a \in \A^*(x) \\
        \ \ 0 & \text{ if not.}
    \end{cases}
\end{equation*}

The Q-learning algorithm is devised to approximate this optimal action-value function $Q^*$. 

\begin{defn}\label{defn:Qlearningpolicy}
    Fix $\varepsilon \in [0,1]$ and $\delta > 0$ the exploration parameter and the step-size respectively. Consider also the initial conditions $x_0 \in \SS$ and $\widehat{Q}_0:\SS \times \A \to \R$. The asynchronous Q-learning with constant step size $\delta$ is the sequence $(X_n,A_n,\widehat{Q}_n)_{n\geq 0}$ where $X_0 = x_0$ and for each $n \geq 0$: 
    \begin{align*}
        &\varphi ( \ a \dado \widehat{Q}_n (X_n, \cdot ) )=  \begin{cases}
        \dfrac{1-\eps}{| \widehat{\A}_n |} + \eps /2  & \text{ if } a \in \widehat{\A}_n. \\
        \ \ \eps/2 & \text{ if not}
    \end{cases}
    \end{align*}
    where $\widehat{\A}_n \defeq \argmax_{a' \in A} \widehat{Q}_n(X_n,a')$. Finally the update rule is given by
    \begin{align*}
    &\widehat{Q}_{n+1} (X_n, A_n) =  (1-\delta) \widehat{Q}_n( X_n, A_n) +  \delta \left( c(X_n,A_n,X_{n+1}) + \gamma \min_{a' \in A} \widehat{Q}_n(X_{n+1},a') \right)  
    \end{align*}
\end{defn}

This procedure can be called $\varepsilon$-greedy asynchronous Q-learning with constant step-size $\delta$. In \cite{Beck2012} (see Theorem 3.4),  the authors derive an upper bound (depending on $\delta$) for the expected error of the approximation $\E( || \widehat{Q}_k - Q^* || _ \infty )$ in the case of finite state spaces.

In the context of Model \ref{Model} we are going to assume that the cost takes the form : 
\begin{equation}\label{eq:cost}
    c( x , a , y) = \left( y_1-x_1 \right) + \left( y_2 - x_2 \right).
\end{equation}
We refer this cost as being ``local'', in the sense that it only depends on the transition $y-x$, and not on the actual state $x$. 
\begin{rmk}\label{rmk:Qlearning-LLP}
    For Q-learning (Definition~\ref{defn:Qlearningpolicy}) to be an LLP, it suffices to use the local cost function of \eqref{eq:cost} and a state-invariant initial condition $\widehat{Q}_0$. 
\end{rmk} 
    
    Finally we see that for the cost function given by \eqref{eq:cost} the total cost $J = J_1$ defined in \eqref{eq:discountedcost} takes the form
    \begin{align*}
        J (x) &= \limsup_n \E \left( \sum_{i=0}^{n-1} c(X_{n+1},A_n,X_n) \ | \ X_0 =x  \right) \\ &= \limsup_n \E \left( |X_{n}| -|X_{0}| \ | \ X_0 = x\right)
    \end{align*}
    where $|x| = x_1+x_2$. So that for this cost, in addition of penalizing the growth of the queues, having finite total cost, i.e. $J < + \infty$, prevents $(X_n)_{n\geq 0}$ of being transient (Definition \ref{defn:unstable}). For an agent always applying ``green'' action we have the following result, where the details of the proof are given in Appendix \ref{appendix:loadbalancing}.

   \begin{lema}\label{rmk: O(1)}
    In the context of Model \ref{Model} let $(X_n^g)_{n\geq 0}$ defined by $X_{n+1}^g \sim p(\ \cdot \ | \ X_n ^g , g )$ for all $n \geq 0$. Consider the notation $f(\gamma) = O ( g ( \gamma) )$ if $f(\gamma) \leq K' g(\gamma)$ then we have that
    \begin{align}
        &J_\gamma ^g (x) =  \E \left(  \sum_{i=0} ^ {+\infty} \gamma^i \left(|X^g_{i+1}|-|X^g_{i}|\right) \ |  \ X_0 = x  \right) = O(1). \label{eq: O(1)} 
    \end{align}
   \end{lema}

\section{Model \ref{Model2}: the free process}\label{sec:model2}


In this section, we present Model \ref{Model2}, which can be seen as a  version of Model \ref{Model} on the whole grid \(\mathbb{Z}^2\). The main goal of this extension is to reduce the impact of boundary states.  To this end, we expand the transition dynamics, initially limited to the positive quadrant, to cover the entire grid.
This simplifies the analysis by eliminating boundary constraints but also offers intrinsic value by creating direct links with models of Self-Interacting Random Walks (see \cite{benjamini-wilson, zerner06, menshikov12} and references therein).


We denote by $\PP(\Z^2)$ the set of probability measures on $\Z^2$. For any $\mu \in \PP(\Z^2)$ s.t. $\E_\mu ( || \xi ||) < + \infty$ we define its drift by $ 
    d(\mu) = \E_\mu ( \xi) \in \R^2.$
\begin{model}\label{Model2}
Consider the triple $(\SS,\A,p)$ with $\SS = {\Z}^2$, $\A = \{r,g\}$ and assume that the transitions probabilities $p( \ y \dado x , a )$  are defined by 
    \begin{equation}
        p (\  x + \xi \dado x, a ) = \mu_a ( \xi ) \text{ for all } \xi \in \Z^2 \text{ and } a \in \{ r,g\},\label{eq:model2 1}
    \end{equation}
    where $\mu_r,\mu_g \in \PP(\Z^2)$ are two fixed measures.
    
    \noindent  
        We further assume that for some $c>0$, 
        \begin{equation}\label{eq: model2 3}
            \E_{\mu_a} ( \exp ( c ||\xi || ) ) < +\infty \quad \text{ for } a=r \ \text{and} \ a=g.
        \end{equation}        
        and that, there exists a unit vector $l \in \R^2$ such that
         \begin{equation}
         \varrho := \min\{ \langle d(\mu_r), l \rangle, \langle d(\mu_g), l \rangle \} > 0 \label{eq:model2 2}
         \end{equation}
          
\end{model}


\begin{rmk}\label{rmk:freeprocess}
    Observe that each $\MM =(\mathcal{S}, \mathcal{A}, p)$ under the assumptions of Model \ref{Model} defines \( \MM' = (\mathcal{S}', \mathcal{A}', p')\) that satisfies the assumptions of Model \ref{Model2}. This is achieved by considering \(\mathcal{S}' = \mathbb{Z}^2\), \(\mathcal{A}' = \{r, g\}\), and setting the transition probabilities as \(p'(x \pm e_i \mid x, a) = \mu_a(\pm e_i)\).
In this sense, we can say that Model \ref{Model2} effectively eliminates the boundary effects present in Model \ref{Model}.

An LLP $(X_n,A_n,\beta_n)_{n\geq 0}$ in $\M$ is determined by an initial state $X_0 \in \Z^2_+$, an initial environment $\beta_0 : \Z^2_{+} \times \{r,g\} \to \R$, a decision rule $\varphi$ and an update rule $\psi$ (see Definition \ref{defn:learningprocess}). By letting $\beta'_0(x,a) = \beta_0(0,a)$ for all $x \in \Z^2$ one extends the initial environment $\beta_0$ to $\Z^2 \times \{r,g\}$ (recall by Definition \ref{defn:learningprocess} the initial environment $\beta_0$ is space invariant). Following this procedure one can verify that the conditions of Definition \ref{defn:learningprocess} are satisfied so that we obtain an LLP $(X_n',A_n',\beta_n ')_{n \geq 0}$ on $\M'$ which we call the \emph{free process} associated to $(X_n,A_n,\beta_n)_{n\geq 0}$. 
\end{rmk}

We present now in Theorem \ref{thm:learningdrift} an important property of LLPs in Model \ref{Model2} which is instrumental for proving our Theorem \ref{thm:MainResult}. Briefly Theorem \ref{thm:learningdrift} states that for any LLP $(X_n,A_n,\beta_n)_{n\geq 0}$ in Model \ref{Model2} there is $L \in \R^2$ such that $\lim_n X_n / n = L $  a.s. and it also characterizes $L = \alpha d(\mu_r)+(1-\alpha)d(\mu_g)$ where $\alpha$ is the average number of times action $r$ is chosen.  For the free process associated to an LLP in Example 1 or 2, we can picture this limiting drift $L$ in Figure \ref{fig:cond model 1} as laying in the dotted line between $d(\mu_r)$ and $d(\mu_g)$. The proof of Theorem \ref{thm:learningdrift} is given in Appendix \ref{sec:proofthm}.

\begin{theorem}\label{thm:learningdrift}
    For any LLP $(X_n, A_n, \beta_n)_{n\geq 0}$  in Model \ref{Model2}, the limit 
    \begin{equation}\label{alpha1}
    \alpha:= 
    \lim_{n \to +\infty} \frac{1}{n} \sum_{i=1} ^n \mathbf{I} ( A_i = r )
    \end{equation}
    exists a.s. and is constant. Therefore, 
    \begin{equation}
        \lim_{n \to +\infty} \dfrac{X_n}{n} = \alpha d(\mu_r) + (1-\alpha) d(\mu_g)
        \ \ \text{a.s.} \label{eq:learningdrift1}
    \end{equation}
\end{theorem}
 
 Let us provide a few key ideas for the proof of this Theorem. We first derive a regenerative structure for a general class of processes (see Appendix \ref{appendix:mathematicalframework}), that contain LLPs in Model \ref{Model2} (see Appendix \ref{appendix:conditions1-3forLLP}). This regenerative  structure implies the Law of Large Numbers $\lim_n X_n/n = L \text{ a.s.}$ for the macroscopic drift of the process  (see Proposition \ref{prop3}) and, in particular, for LLPs in Model \ref{Model2}. Then it follows that $\alpha$ in \eqref{alpha1} is well defined and that the characterization \eqref{eq:learningdrift1} holds (Appendix \ref{sec:proofthm}).



\section{Main result for LLPs}\label{sec:mainresult}

In this section we present our main result Theorem~\ref{thm:MainResult}, where we show that an agent operating under the Local Learning Process (LLP) can lead to transient behavior. Consequently, we show that such an agent also performs sub-optimally in terms of the total discounted cost. Importantly, this setting includes agents employing asynchronous Q-learning algorithms with a constant step size (see Remark \ref{rmk:Qlearning-LLP}) in queueing problems (e.g. Example 1 and 2), thereby highlighting the practical implications of our findings for commonly used learning algorithms.

This result can be interpreted as follows. 
From the description of LLP, see Definition~\ref{defn:learningprocess},  there is a policy that is applied in every state $x$ that is visited for the first time.  Because of the lack of knowledge of the underlying dynamics, this starting policy might be transient. We will then show that the the LLP agent is unable to recover from this instability, leading to a transient behavior.




Introduce the following notation: 
\begin{align*}
    f(\gamma) = \Omega(g(\gamma)) \ 
\text{ if } \ k g(\gamma) \leq f(\gamma) \leq K g(\gamma),
\end{align*}
for constants $k,K > 0$ and for $\gamma \in (1-\varepsilon,1)$, where $\varepsilon$ is sufficiently small. The detailed proof of the following theorem is provided in Appendix \ref{sec:proofMainresult}.

\begin{thm}\label{thm:MainResult}
    Consider $(\SS, \A, p )$ as in Model \ref{Model}. Then there are $0<q_0<q_1\leq 1$ such that any LLP in $(\SS, \A, p )$ with $q  := \varphi (\ r \dado \beta_0 ( x , \cdot ) ) \in (q_0,q_1)$ is transient.
    Also, 
    \begin{equation} \label{eq:Loss function}
        J_\gamma(x) = \Omega\left( \frac{1}{1-\gamma}\right) \quad \text{ for all } x \in \SS.
    \end{equation}
\end{thm}

Under the assumptions of Theorem \ref{thm:MainResult}, any LLP exhibits sub-optimality as $\gamma \to 1$. This sub-optimality follows because the optimal discounted cost satisfies the inequality $ J^*_\gamma(x) \leq J^g_\gamma (x) = O(1)$ (Lemma \eqref{rmk: O(1)})  while from \eqref{eq:Loss function} we obtain $J_\gamma (x) = \Omega (1/(1-\gamma)) \to +\infty$ as $\gamma \to 1$.


\begin{rmk}
As an LLP in Model~\ref{Model} (Remark~\ref{rmk:Qlearning-LLP}), Q-learning is subject to Theorem~\ref{thm:MainResult} under a poor choice of initial condition $\widehat{Q}_0$ and/or exploration parameter $\varepsilon$, leading to instability and suboptimality. Numerical experiments (see Section~\ref{sec:numerics}) suggest that this instability is a more general phenomenon. It persists even in settings where Q-learning does not fit the LLP framework, for instance, when using a nonlocal cost function (such as the sum of the coordinates of the current state) or a nonconstant step-size.
\end{rmk}

We now outline the proof of Theorem~\ref{thm:MainResult}; a complete version is provided in Appendix~\ref{sec:proofMainresult}.

\noindent
{\bf Main steps of the proof.}
\begin{enumerate}
    \item Recall Remark~\ref{rmk:freeprocess} where we show that any LLP in Model~\ref{Model} naturally induces an LLP in Model~\ref{Model2}, which we refer to as the associated \emph{free process}.
    
    \item According to Theorem \ref{thm:learningdrift}, the free process has a limiting drift $L$ characterized by both drifts $d(\mu_r)$ and $d(\mu_g)$, and by $\alpha$,  the average number of times the agent applies action $r$. We show in Lemma \ref{cor:lowerbound} how to control $\alpha$ in terms of the parameters of the model, and we use this to show that $L \in \R^2_+$ under the assumptions of Theorem \ref{thm:MainResult} . 

    \item Then we apply the results derived in Appendix 
    \ref{appendix:renewalcones}, to show that our free process possesses a regenerative cycle structure. 
     With this we conclude that there is a random time, call it ``success time'', when the original LLP enters a cone (contained in the interior of the positive orthant) and never leaves it after this time. This establishes the transience part of Theorem \ref{thm:MainResult}.
    \item Finally after this ``success time'' the original LLP behaves exactly as the associated free process conditioned to stay in this cone. In particular the regenerative structure from Appendix \ref{appendix:renewalcones} holds for the original LLP, too. This regenerative structure enables us to show \eqref{eq:Loss function}, that completes the proof of our Theorem \ref{thm:MainResult}.    

\end{enumerate}

\section{Numerics}\label{sec:numerics}


    In this section we perform numerical experiments in the load balancing problem (Example 1) to illustrate the behavior of Q-learning in various settings combining constant and nonconstant step size and local and nonlocal cost function. As nonconstant step size  we consider $\delta_n = 1/(n+1)$, which satisfies the classical Robbins-Monro conditions. As a nonlocal cost function we consider the total number of jobs, i.e. $ c( x , a , x') = x_1 + x_2$, which is a standard objective  considered in the literature. In particular, we consider the following four settings:
    (i) (setting of Theorem~\ref{thm:MainResult}) constant step-size and local cost function as given by \eqref{eq:cost}, (ii) constant step-size and nonlocal cost function, (iii) nonconstant step-size and local cost function, and (iv) nonconstant step-size and nonlocal cost function.  

\begin{figure}[ht]
    \centering
           \includegraphics[width=0.7\linewidth]{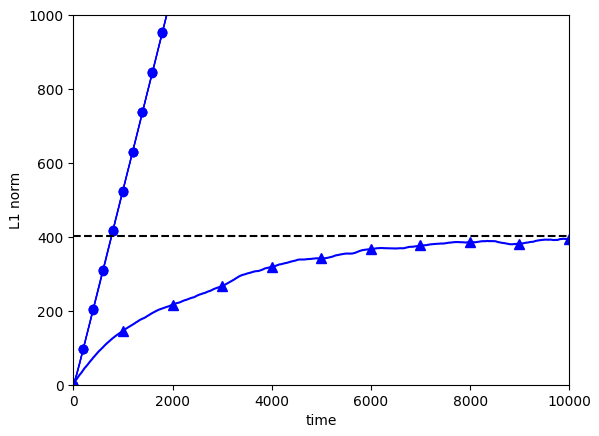}
           \caption{Evolution over time of the average L1-norm for Q-learning in settings (i)-(iv) (circles) and ``green'' policy (triangles). The dashed line indicates the value $J ^g$. }
           \label{fig: L1norm}
\end{figure}

\begin{figure}[ht]
    \centering
           \includegraphics[width=0.7\linewidth]{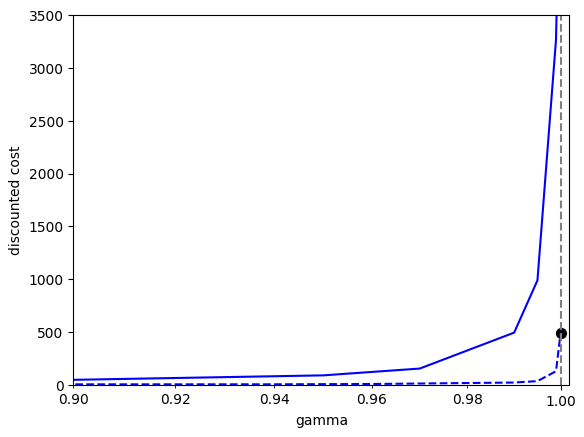}
           \caption{The discounted cost $J_\gamma$ for the Q-learning algorithm (solid) and $J_\gamma ^g$ (dashed) as a function of $\gamma$.}
           \label{fig: discounted cost}
\end{figure}

Throughout the experiments, we fix the following choice of parameters for the load balancing example (Example 1): $\lambda = 1.5, \mu_1 = 0.1, \ \mu_2 = 0.35, \ p_r = 0.45,$ $p_g = 0.8$ and $\widetilde{\mu} = 10.8$. For the parameters of Q-learning we consider the discounting factor $\gamma = 0.1$, exploration parameter $\eps = 0.1$, step-size $\alpha = 0.2$, initial condition $X_0 = (0,0)$ and initialization $\widehat{Q}_0 ( x, a ) = 0 $.


In Figure~\ref{fig: L1norm}, we simulate $N = 1000$ trajectories up to time $n = 10000$ and we plot over time the average  $l_1$ norm $|x| = x_1+x_2$ for $(X_n^g)_{n \geq 0}$ and $(X_n)_{n\geq 0}$ derived from Q-learning for settings (i)-(iv) with initial condition $X_0 = (0,0)$. The dashed line indicates the value $J^g((0,0))$. We observe that the curve for the ``green'' policy converges to $J^g((0,0))$ due to ergodicity  (recall Lemma~\ref{lemma:lyapunovgreen}).
The four curves for Q-learning in the settings (i)-(iv) are visually indistinguishable, suggesting that Q-learning remains transient even when using decreasing step-sizes and nonlocal costs.

In Figure~\ref{fig: discounted cost}, under the setting of Theorem~\ref{thm:MainResult}, we plot the discounted cost functions $J_\gamma(X_0)$ and $J_\gamma ^g(X_0)$ as a function of 
 $\gamma \in [0.9,1]$ for fixed $X_0 = (0,0)$. For this we simulate $N = 100$ values of $J_\gamma$, calculated for a $n = 10000$ path-length, and plot an average for each $\gamma$. We recall that from Remark~~\eqref{rmk: O(1)}, we know that as $\gamma \to 1$ (by dominated convergence theorem) $J_\gamma ^g \to J^g = \E \left(  |X_\infty^g| \right) - |X_0|$, a value we represent with a  big circle  in Figure~\ref{fig: discounted cost}. As predicted by Theorem~\ref{thm:MainResult}, we observe that $J_\gamma$ increases with $\gamma$ and has an asymptote at $\gamma=1$, whereas 
 $J_\gamma ^g$ stays finite (Lemma \ref{rmk: O(1)}).

\section{Conclusion, extensions and future work}\label{sec:conclusion}

Our main result (Theorem~\ref{thm:MainResult}) establishes that locally operating learning algorithms—formally modeled here as Local Learning Processes (LLPs)—exhibit a crucial sensitivity to initial conditions in the context of infinite state spaces. We demonstrate that an initial condition leading to transient and suboptimal behavior generates a trajectory from which these local algorithms cannot recover.
This finding implies that, in infinite-state settings, ensuring favorable initial conditions for local learning is not merely beneficial but essential. A critical direction for future research is therefore the development of methods to systematically identify such stable initial configurations. Once identified, these configurations pave the way for rigorous theoretical guarantees on subsequent algorithmic convergence.
This latter challenge is of broad interest, as it applies to both local and nonlocal algorithms. Convergence guarantees for model-free algorithms in infinite state spaces remain an open and poorly understood area.

We finally note that an interesting challenge stemming from our work, is to adapt  our analytical framework, in particular Appendix~\ref{appendix:mathematicalframework}, to accommodate  models with heavy-tailed distributions of jumps.

\section*{Acknowledgments}

The authors would like to thank Sean Meyn for valuable discussions. Research was partially supported by the French "Agence Nationale de la Recherche (ANR)" through the project ANR-22-CE25-0013-02 (ANR EPLER).

\printbibliography

\clearpage
\begin{table}[h] \label{glossary}
\centering
\begin{tabular}{>{}p{0.11\textwidth} | p{0.8\textwidth}}
\hline
 \textbf{Symbol }& \textbf{Definition}  \\
\hline
$\SS, \A, p$ & State space, action space and transition probabilities.  \\
\hline
$\mu_a, d(\mu_a)$ & The law of jumps for action $a$ in interior points and its  corresponding drift.  \\
\hline
$l$ & Common vector for drifts. \\
\hline
$\varrho$ & Magnitude of the minimal component between drifts and vector $l$. \\
\hline
$\alpha$ &  Average number of times the ``red'' action is selected.\\
\hline
$X_n,A_n,\beta_n$ & State, action and environment at time $n$.  \\
\hline
$S_n$ & Projection of $X_n$ onto $l$.  \\
\hline
$L$ & Limiting drift of $(X_n)_{n\geq 0}$.\\
\hline
$\R^d, \R^d_+$ & The space of all $d$-dimensional real-valued vectors; its restriction to vectors with nonnegative coordinates. \\
\hline
$\Z^d$, $\Z^d_+$ & The space of all $d$-dimensional integer-valued vectors; its restriction to vectors with nonnegative coordinates. \\
\hline
$\langle \cdot , \cdot \rangle $, $||\cdot||$ & The usual euclidean inner product and norm. \\
\hline
$|\cdot|$ & The $l_1$ norm of a vector.  \\
\hline
\bottomrule
\end{tabular}
\caption{}
\label{tab:glossary}
\end{table}

\appendix

\section{Mathematical framework}\label{appendix:mathematicalframework}

In this Appendix, we derive the main mathematical tools used in the proof of Theorem~\ref{thm:MainResult}. These tools are developed for a class of processes in $\R^d$ 
 defined by Conditions 1–3 -- a class that is more general than the one of LLPs in Model~\ref{Model2} (see Appendix~\ref{appendix:conditions1-3forLLP}).

We begin in Appendix~\ref{appendix:stochasticassumptions}, where we fix notation and introduce Conditions 1 and 2. Then, in Appendix~\ref{appendix:strictsubmartingales}, we present several preliminary consequences of these conditions, with Proposition~\ref{prop: finite Tk} playing a particularly important role.

Next, in Appendix~\ref{appendix:renewal}, we introduce Condition 3 and prove in Proposition~\ref{prop3} that Conditions 1–3 together imply a regenerative structure for the class of stochastic processes under consideration. We also show that for each process in this class there is a vector $L$, such that the process drifts  to infinity along direction $L$.

Finally, Appendix~\ref{appendix:renewalcones} addresses the existence of a more refined regenerative structure in which 
$L$ lies in the positive orthant. This structure is a central component in the proof of Theorem~\ref{thm:MainResult}.

\subsection{Stochastic assumptions}\label{appendix:stochasticassumptions}
Let $(X_n)_{n\geq 0}$ be a random sequence taking values in $\R^d$ where $d \geq 1$ (note that our main result Theorem \ref{thm:MainResult} only requires the analysis for $d=2$). Here $X_0\in \R^d$ is an initial state that may be random.
 
Let ${\cal F}_n$ be an increasing sequence of sigma-algebras such that $ \sigma(\{X_0,\ldots,X_n\}) \subseteq {\cal F}_n$ for any $n \geq 0$.

{\bf Condition 1.} There exists a unit vector $l\in \R^d$ such that the scalar products $
S^{(l)}_n := \langle X_n,l\rangle$ have a uniformly positive drift: 
\begin{align}\label{G0}
\varrho: = \essinf \inf_{n \geq 0} \mathbb{E} \left(S^{(l)}_{n+1} - S_n^{(l)} \ | \ {\cal F}_n \right)  > 0 
\end{align}
i.e., $(S_n ^{(l)})_{n\geq 0}$ forms a strict submartingale (w.r.t. $\{{\cal F}_n\}$).

The quantity $\varrho$ defined by \eqref{G0} can be linked to \eqref{eq:model2 2} in the context of LLPs in Model \ref{Model2}. We refer to Appendix \ref{appendix:conditions1-3forLLP} for details.

To simplify the notation, we assume that $l$ is fixed and omit, for short, the upper index $l$ writing $S_n:=S_n^{(l)}$. Further, we let
$\xi_n = X_{n}-X_{n-1}, n\ge 1$ and $\psi_n = \langle \xi_n, l \rangle$.  

In what follows, it is convenient to use  the following representation. Let ${\cal G} = \{G\}$ be a measurable space of probability distributions on ${\R}^d$ with sigma-algebra ${\cal B_G}$ such that, for any $G\in {\cal G}$, 
\begin{align}\label{G1}
\int_{{\mathbb R}^d} \langle x,l \rangle dG(x) \in [\varrho,\infty)
\end{align}
Then, for any $n$, given ${\cal F}_n$, the random variable $\xi_{n+1}$ has a distribution belonging to ${\cal G}$. More precisely, the random variable $\xi_{n+1}$ has a distribution $G_{n+1}$ that is a measure-valued random variable taking valued in ${\cal G}$, that is measurable w.r.t. ${\cal F}_n$. Further, given $G_{n+1} = G\in {\cal G}$, the random variable $\xi_{n+1}$ does not depend on ${\cal F}_n$. 

{\bf Notation.} For any measurable function $f$ and any distribution $G\in {\cal G}$,
we let \begin{align*}
&{\mathbb P}_G (f(\xi) \in \cdot) := {\mathbb P}
(f(\xi_{n})\in \cdot \ | \ G_{n} = G)  \text{ and} \\ 
&{\mathbb E}_G (f(\xi)) := {\mathbb E}
(f(\xi_{n}) \ | \ G_{n} = G), \hfill 
\end{align*} where the right-hand sides do not depend on $n$.
Then 
\eqref{G0} may be rewritten as
\begin{equation}\label{G3}
\varrho:= \inf_{G} {\mathbb E}_G \left( \langle \xi, l \rangle \right) > 0.
\end{equation}

\textbf{Condition 2.} 
We assume that there exists $c>0$ such that 
\begin{equation}\label{G4}
\sup_{G \in \G} {\mathbb E}_G (\exp \left({c||\xi||}\right)) < \infty.
\end{equation}
It follows from \eqref{G4} that
\begin{equation}\label{Ge}
\widehat\varrho:=\sup_{G \in \G} {\mathbb E}_G ( \langle \xi, l \rangle ) < \infty
\end{equation}

We note that LLPs in Model~\ref{Model2} can be cast within this framework. In that setting, 
$G = \mu_a$ for $a \in \{r,g\}$, \eqref{G1} is derived from \eqref{eq:model2 2} and \eqref{eq: model2 3}, and \eqref{G4} is exactly \eqref{eq: model2 3}. Further details are provided in Appendix \ref{appendix:conditions1-3forLLP}.

\begin{remark}\label{UI}
One can show that \eqref{G4} also implies that the sequences of random variables $\{S_n/n\}$ is uniformly integrable. This will be formally justified later in the context of Model~\ref{Model2}.
\end{remark}

\begin{remark}
    Condition~\eqref{G4} may be relaxed; however, the existence of exponential moments significantly simplifies the proofs that follow.
\end{remark}

\subsection{Strict submartingales: Results and Proofs}\label{appendix:strictsubmartingales}

Our analysis is grounded in the following mathematical framework. We begin by employing fundamental martingale techniques to establish a key preliminary construction. Specifically, we consider a strict submartingale in 
${\mathbb R}^d$
 with increments following light-tailed distributions.

Assuming Conditions 1 and 2, we show in Proposition~\ref{prop: finite Tk} that there exist infinitely many time instants $T_k$ 
with the following properties:\\
First, each $T_k$ 
 is a “record time,” meaning that the projection of the trajectory onto the common direction $L$ 
remains below its value at time $T_k$
 for all earlier times. Second, after time $T_k$,
the projection stays above its value at $T_k$.

Consequently, the trajectory is naturally partitioned into intervals between successive record times. We then demonstrate that the lengths of these intervals have finite exponential moments.\\

We assume Conditions 1 and 2 to hold throughout the rest of this section.

\begin{proposition}\label{prop1}
For any $0\leq \varepsilon < \varrho$ there exist positive constants $c_1=c_1(\eps)$ and $c_2=c_2(\eps)$ such that, for any $x\ge 0$ and for all $n=0,1,\ldots$, 
\begin{align}
p_{n,x}(\eps) &:= {\mathbb P}(S_n-S_0 \le  \eps n  -x \ | \ {\cal F}_0) \notag  \\ &\leq
\exp(-c_1x-c_2n) \quad \text{a.s.} \label{P1}
\end{align}

There are also positive constants $c_3$ and $c_4$ such that, for any $x\ge 0$ and for all $n=0,1,\ldots$
\begin{align}
 & {\mathbb P}(S_n-S_0 \ge 2n\widehat{\varrho}+x \ | \ {\cal F}_0)  \leq
\exp(-c_3x-c_4n) \quad \text{a.s.}
\label{P100}
\end{align}
\end{proposition}
\begin{proof}
The proofs of \eqref{P1} and \eqref{P100} are similar, we prove \eqref{P1} only. 
We apply Chernoff's (that is ``exponential Markov'') inequality:
\begin{align}
& \P( S_0 -S_n \geq x - \eps n \ | \ \F_0 ) \notag \\
& \leq  \E( \exp ( t(S_0 - S_n) )\ | \ \F_0 ) \exp ( t \eps n - t x) \notag \\
& = \E \left( \prod_{i=1}^n \exp ( -t(S_i - S_{i-1}) \ | \ \F_0 \right) \exp ( t \eps n - t x) \notag \\
& \leq \left( \sup_G \E_G ( \exp ( -t \langle \xi , l \rangle  ) ) \right)^n \exp ( t \eps n - t x) = \notag \\
& \leq \left( \sup_G \exp ( F_G(-t) ) \right)^n \exp ( t \eps n - t x), \label{eq:chernov}
\end{align}

where $F_G(t) = \log \E_G ( t \langle \xi, l \rangle)$. Condition 2 implies that there is $t_0 > 0$ such that $F_G ( t )$ is well defined for all $t \in (-t_0,t_0)$ and all $G \in \mathcal{G}$. Moreover we can consider $t_0$ such that for all $t \in (0,t_0)$
        \begin{equation}\label{eq:F1}
        F_G(t) \leq 2 \widehat{\varrho}
    \end{equation}
    and for each $\varrho' < \varrho$ there is $0< t < t_0$ such that
        \begin{equation}\label{eq:F2}
        F_G(-t) < -t \varrho'.
    \end{equation} We conclude \eqref{P1} from the inequality \eqref{eq:chernov} and \eqref{eq:F2}. 
\end{proof}

\begin{cor}\label{cor1}
There exist positive constants $c_1$ and $C$ such that, for any $x>0$ and for any $m=0,1,\ldots$ 
\begin{align}
 p^{(m)}(x) &:= {\mathbb P}(\inf_n(S_{m+n}-S_m) < -x \ | {\cal F}_m)   \notag \\
& \leq C \exp (-c_1x), \quad \text{a.s.} \label{P2} 
\end{align}
\end{cor}
\begin{proof}
It is enough to consider the case $m=0$. Using the union bound and Proposition \ref{prop1} we obtain
\begin{align*}
p^{(0)} (x) &\leq\sum_n p_{n,x} (0) \\
& \leq \sum_n \exp (-c_1 (0) x -c_2(0)n) \\
& = C \exp (-c_1 x) \quad \text{a.s.} 
\end{align*}
where $c_1 = c_1(0)$ and $C = \sum_n \exp (-c_2 (0) n) $. 
\end{proof}
For $m=0,1,\ldots$, introduce the events
\begin{equation}\label{Dm}
D_m = \left\{ \inf_{k\geq 1} (S_{m+k}-S_m)>0 \right\}.
\end{equation}
\begin{cor}\label{cor2}
There exists $\varepsilon >0$, such that, for any $m=0,1,\ldots$, 
\begin{equation}\label{P3}
P_m:= {\mathbb P} (D_m \ | \ {\cal F}_m)\ge \varepsilon \quad \text{a.s.}
\end{equation}
\end{cor}
\begin{proof}
It is enough to consider the case $m=0$.  Recall the notation $\psi_n = <\xi_n,l>$. 
Let $x_0\ge 0$ be such that $p^{(0)}(x_0)<1$ (this can be done by Corollary \ref{cor1})
and let $j\ge 1$ be such that
\begin{equation*}
\Delta := \inf_G {\mathbb P}_G (\psi > x_0/j) >0.
\end{equation*}
Then if $U_j = \inf_k (S_{j+k}-S_j)$ we have that $D_0 \supset \{ \psi_1 > x_0/j, \ldots,
\psi_j > x_0/j, U_j \ge -x_0 \} $. Then
\begin{align*}
P_0 &\geq 
{\mathbb P} (\psi_1 > x_0/j, \ldots,
\psi_j > x_0/j, U_j \ge -x_0
\ | \ {\cal F}_0)\\
&\geq 
\Delta^j \cdot  \essinf 
{\mathbb P} ( U_j \ge -x_0
\ | \ {\cal F}_j) \\
&\geq
\Delta^j (1-p^{(0)}(x_0)) =: \varepsilon >0, \quad
a.s.
\end{align*} where the inequality from the second to the third line is justified by the previous corollary. \qed
\end{proof}

From the previous corollary it follows that
\begin{equation}\label{Phat}
    \widehat{P}:= \inf_m P_m  \geq \varepsilon \quad \text{a.s.}
\end{equation}
Consider $M_n = \{ S_n > S_k \text{ for all } k < n\}$ and let 
\begin{equation}\label{T1}
T_1 = \inf \{n\geq 0: {\mathbf I}(D_n \cap M_n) = 1 \}\leq \infty
\end{equation}
so that at the random time $T_1 \geq 0$ the process $(S_{n})_{n\geq 0}$ attains a maximum value and never goes below that value in the future. Inductively we define for $k=1,2,\ldots$ if $T_k<\infty$, the random time 
\begin{equation}
T_{k+1}= \inf \{ n > T_k : \mathbf{I} ( D_n \cap M_n ) = 1 \}\leq \infty. 
\end{equation}

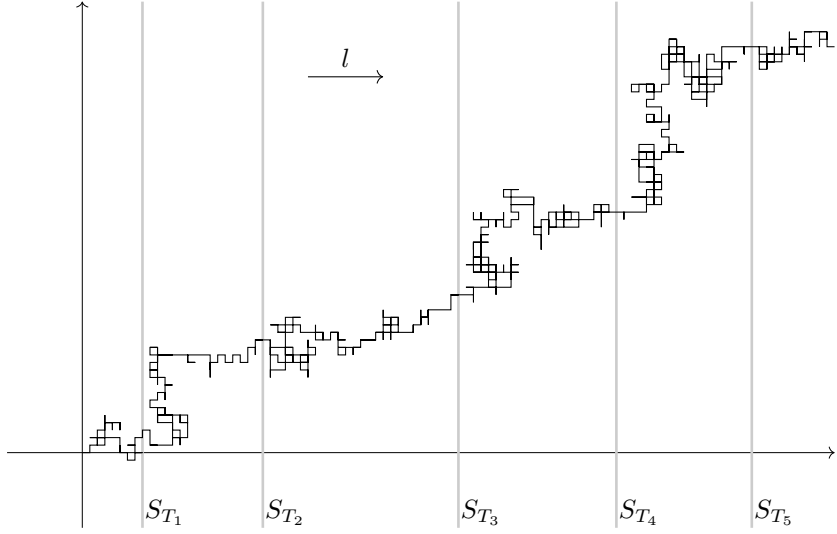
\begin{figure}[ht]
    \centering
\pgfmathsetseed{321618684} 

\newcounter{step}
\def\pathdirections{}
\setcounter{step}{0}
\loop
    \stepcounter{step}
    \pgfmathrandominteger{\direction}{0}{9}
    \ifnum\thestep=1
        \edef\pathdirections{\direction}
    \else
        \edef\pathdirections{\pathdirections,\direction}
    \fi
\ifnum\thestep<954\repeat

\begin{tikzpicture}[scale=0.1]
    \draw[->, line width=0.25pt, black] (-10,0) -- (100,0) ;
    \draw[->, line width=0.25pt, black] (0,-10) -- (0,60) ;
    \draw[-, line width=1pt, gray!40] (8,-10) -- (8,60);
\node[black] at (11,-8)  {$S_{T_1}$};
\draw[-, line width=1pt, gray!40] (24,-10) -- (24,60);
\node[black] at (27,-8)  {$S_{T_2}$};
    \draw[-, line width=1pt, gray!40] (50,-10) -- (50,60);
    \node[black] at (53,-8)  {$S_{T_3}$};
    \draw[-, line width=1pt, gray!40] (71,-10) -- (71,60);
\node[black] at (74,-8)  {$S_{T_4}$};
\draw[-, line width=1pt, gray!40] (89,-10) -- (89,60);
\node[black] at (92,-8)  {$S_{T_5}$};
\draw[black, ->] (30,50) -- (40,50) node[above, midway] {$l$};
    \draw [thin, black] (0,0) \foreach \d in \pathdirections {
        \ifnum \d <3
            -- ++(1,0) 
        \else
            \ifnum \d <6
                -- ++(0,1) 
            \else
                \ifnum \d <8
                    -- ++(-1,0) 
                \else
                    -- ++(0,-1) 
                \fi
            \fi
        \fi
    };

\end{tikzpicture}
\caption{Representation of a sample path $(X_n)_{n\geq 0}$ and the levels $S_{T_k}$ in direction $l$ .}
\end{figure}

\begin{proposition}\label{prop: finite Tk}
The random variables $T_1,T_2-T_1,\dots,T_{k+1}-T_k,\dots,$ are a.s. finite and, moreover, they uniformly possess a finite exponential moment, i.e. there are finite constants $c>0$ and $C>0$ such that 
\begin{align}\label{Tk}
&{\mathbb E} \left( \exp (cT_1) \ | \ 
{\cal F}_0\right)\leq C \quad\text{and} \notag \\ 
& {\mathbb E} 
\left( \exp (c(T_{k+1}-T_{k})) \ | \ {\cal F}_{T_k}\right) \le C  \ \ \text{a.s. for all } k \geq 1.
\end{align}
\end{proposition}
\begin{proof}
In {\bf Step 1}, we prove the first
inequality in \eqref{Tk}, and then in {\bf Step 2} the second inequality.

{\bf Step 1}.
Without loss of generality, we may assume that $S_0=0$ a.s.
 Let
\begin{equation*}
t_1 = \inf \{n\ge 1: S_n \le 0\}\le \infty
\end{equation*}
and, on the event $\{t_1<\infty\}$, let
\begin{equation*}
v_1 = \max_{0\le i <t_1} S_i \quad
\text{and} \quad
u_1 = \inf \{i>0: S_{t_1+i} > v_1 \}.
\end{equation*}
For $k=1,2, \ldots$, we use an induction argument. Let 
\begin{equation*}
W_0=0 \ \ \text{and} \ \ W_k = \sum_{i=1}^k (t_i+u_i).
\end{equation*} 
Given $W_k < \infty$, we let  
 \begin{equation*}
t_{k+1} = \inf \{n\ge 1:  S_{W_k+n} \leq S_{W_k}\} \le \infty
\end{equation*}
and, on the event $\{W_k<\infty\}\cap\{t_{k+1} <\infty\}$, we let
\begin{align*}
&v_{k+1} = \max_{W_k\le i <W_k+t_{k+1}} S_i \quad
\text{and} \\
&u_{k+1} = \inf \{m>0: S_{W_k+t_{k+1}+m} > v_{k+1}\}.
\end{align*}
Let 
\begin{equation*}
\theta = \min \{k : {\textbf I} (D_{W_k})=1 \}.
\end{equation*}
Then we obtain $T_1 = W_\theta$. Observe that the distribution of $\theta$ is bounded from above by a geometric distribution with parameter $\widehat{P}$ given by \eqref{Phat}. This in particular gives us that $T_1$ is a.s. finite and that $T_1=0$ if $\theta=0$. Further, all $T_k, k\geq 2$ are a.s. finite by similar reasons.

{\bf Step 1.1} We prove first that the random variable
$W_1 {\bf I}(W_1<\infty)$ has a uniformly bounded conditional exponential moment, given ${\cal F}_0$.
For that, we show consequently that the tail distributions of $t_1{\bf I}(t_1<\infty)$ , $v_1{\bf I}(t_1<\infty)$,
$-S_{t_1}{\bf I}( t_1<\infty)$ and $u_1{\bf I}(t_1<\infty)$ decay exponentially fast (again uniformly conditioned on ${\cal F}_0$).
Then  the random variable 
\begin{align*}
W_1 {\bf I}(W_1<\infty) \equiv W_1 {\bf I}(t_1<\infty) =
(t_1+u_1) {\bf I}(t_1<\infty)
\end{align*} 
has also an exponentially decaying tail distribution, since, for any $c>0$,
\begin{align*}
 \exp(cW_1 {\bf I}(W_1<\infty))
 \leq \exp(2ct_1{\bf I}(t_1<\infty)) +
 \exp(2cu_1{\bf I}(t_1<\infty)).
\end{align*}

For any $i=0,1,\ldots$, recall the notation $\psi_i = \langle \xi_i, l \rangle$, we have
\begin{align*}
{\mathbb P}(t_1=i+1 \ | \ {\cal F}_0)
&=
{\mathbb P} (\min_{k\le i} S_k >0, S_{i+1}\leq 0 \ | \ {\cal F}_0)\\
&\leq
{\mathbb P} ( S_i >0, S_{i+1}\leq 0 \ | \ {\cal F}_0)\\
&\leq
{\mathbb P} ( S_i \leq i\varrho/2 \ | \ {\cal F}_0)  + 
{\mathbb P} ( S_i\geq i\varrho/2, \psi_{i+1}< - i\varrho/2 \ | \ {\cal F}_0)\\
&\leq
{\mathbb P} ( S_i \leq i\varrho/2 \ | \ {\cal F}_0) + 
{\mathbb P} (\psi_{i+1}< - i\varrho/2 \ | \ {\cal F}_0)\\
&\leq C_1 \exp (-C_2i) \quad \text{a.s.}
\end{align*}
where $C_1$ and $C_2$ are some absolute constants. Then
\begin{align*}
{\mathbb P}(i<t_1<\infty \ | \ {\cal F}_0) 
&= 
\sum_j
{\mathbb P}(t_1=i+j \ | \ {\cal F}_0)\\
&\leq
C_3 \exp (-C_2 i) \quad \text{a.s.}
\end{align*}
where $C_3$ and $C_2$ are some absolute constants.

Next, for any $x>0$,
\begin{align*}
{\mathbb P} (v_1>x, t_1<\infty \ | \ {\cal F}_0) 
&\leq
\inf_i [
{\mathbb P} (\max_{j\leq i} S_j > x \ | \ 
{\cal F}_0)
+ {\mathbb P} (i<t_1<\infty \ | \ {\cal F}_0)]\\
&\leq
{\mathbb P} (\max_{j\leq i_x} S_j > x \ | \ {\cal F}_0)
+ {\mathbb P} (i_x<t_1<\infty \ | \ {\cal F}_0)\\
&\leq 
\widehat{C}_1 \exp (-\widehat{C}_2 x), 
\end{align*}
where $i_x$ may be taken, say, as $i_x = [x/2\widehat{\rho}]$ and $\widehat{C}_1$ and $\widehat{C}_2$ are absolute constants, thanks to (\ref{P100}) in Proposition \ref{prop1}.

Then,  for $x> 0$, 
\begin{align*}
{\mathbb P} (-S_{t_1}{\bf I} (t_1<\infty)>x) 
& =
\sum_{i\ge 0} 
{\mathbb P} (-S_{t_1}>x, t_1=i+1 \ | \ {\cal F}_0)\\
&\leq 
\sum_{i\ge 0} \left(
{\mathbb P}(S_i \geq i\varrho/2 \ | \ {\cal F}_0) \sup_G {\mathbb P}_G (\psi <  -x - i\varrho/2)
\right. \\
& \hspace{30pt} \left.
+ \ {\mathbb P} (S_i \in (0,i\varrho/2) \ | \ {\cal F}_0)\sup_G {\mathbb P}_G (\psi < -x)\right)\\
&\leq \widetilde{C}_1 \exp (-\widetilde{C}_2x),
\end{align*}
where $\widetilde{C}_1$ and 
$\widetilde{C}_2$ are absolute constants.
Indeed, each of the probabilities ${\mathbb P}_G(\dots)$ is bounded from above using the Chernoff's inequality, the probability for $S_i \ge i\varrho/2$ is bounded from above by 1 and the probability for $S_i$ to be below $i\varrho/2$ is bounded by an exponentially decaying function, see Proposition \ref{prop1} again.

Then, for an integer $x>0$ and for any $K>0$, 
\begin{align*}
{\mathbb P} (u_1>x, t_1<\infty \ | \ {\cal F}_0) 
&\leq
{\mathbb P}(v_1>Kx/2,t_1<\infty \ | \ {\cal F}_0) \\
&+ {\mathbb P}(-S_{t_1}>Kx/2,t_1<\infty \ | \ {\cal F}_0) \\
&+
\sum_i {\mathbb P} (t_1=i, \max_{j\leq x}
(S_{i+j}-S_i) \le Kx) \ | \ {\cal F}_0),
\end{align*}
where the first two terms are bounded by exponentially decaying functions and the
latter sum is bounded from above by
$\esssup {\mathbb P} (\max_{j\leq x} (S_j-S_0) \leq Kx \ | \ {\cal F}_0)$
that admits an exponentially decaying upper bound if $K\le \varrho/2$.

Thus, there exist absolute constants $c$ and $C$ such that
\begin{align}\label{W1}
    {\mathbb E} \exp (cW_1{\bf I}(W_1<\infty)
    \ | \ {\cal F}_0) \leq C \ \ \text{a.s.}
\end{align}

{\bf Step 1.2.}
It follows directly that, for any $k\geq 2$ and for the same constants $c$ and $C$, 
\begin{align}\label{Wk}
{\mathbb E} \exp (c(W_{k}-W_{k-1})
{\bf I}(W_{k}-W_{k-1}<\infty) \  | \  {\cal F}_{k-1}, W_{k-1}<\infty) \leq C \ \ \text{a.s.}
\end{align}
Indeed, it is enough to shift the time origin from $0$ to $W_{k-1}$.

Properties \eqref{W1} and  \eqref{Wk} imply that, for any 
smaller value of $c$, all random variables 
$U_k := \exp(c(W_k-W_{k-1})) {\bf I} (W_k-W_{k-1}<\infty)$ are conditionally uniformly integrable, i.e., for any $\delta >0$, there exists a sufficiently large $K_{\delta}$ such that
\begin{align*}
    {\mathbb E} (U_k {\bf I}(U_k>K_{\delta}) \ | \ 
    {\cal F}_{W_{k-1}}, W_{k-1}<\infty) \leq \delta \ \ \text{a.s.}
    \end{align*}
This, in turn, implies that we can take a very small value of $c>0$ such that, for any $k$, 
\begin{align}\label{valuec}
    {\mathbb E} (U_k \ | \ {\cal F}_{W_{k-1}}, W_{k-1}<\infty)\leq 1-\varepsilon/2 =:\Delta <1\ \ \text{a.s.}
\end{align}

{\bf Step 1.3.} We are ready now to prove the first inequality in \eqref{Tk}.
Take a value $c$ satisfying \eqref{valuec}. We have 
\begin{align*}
{\mathbb E}(e^{cT_1} \ | \ {\cal F}_0) &= 
\sum_{k=1}^{\infty}{\mathbb E}(e^{cW_{k}} {\bf I}(\theta =k)) \ | \ {\cal F}_0) \\
&\leq
\sum_{k=1}^{\infty}{\mathbb E}(e^{cW_{k}} {\bf I}(\theta \geq k)) \ | \ {\cal F}_0)\\
&=
\sum_{k=1}^{\infty}
{\mathbb E}
\left(\prod_{j=1}^k e^{c(W_{j}-W_{j-1})}{\bf I}(t_j<\infty) \ | \ {\cal F}_0\right)\\
&\leq \sum_{k=1}^{\infty} \Delta^k
= \frac{\Delta}{1-\Delta} =:C < \infty.
\end{align*}


{\bf Step 2.} We prove now the second inequalities in 
\eqref{Tk}.

Take any $k\ge 1$ and consider $T_{k+1}-T_k$ conditioned on the event $T_k=m$, for some $m\geq k$.  
Given that,  
\begin{align*}
 T_{k+1}-T_k &= T_{k+1}-m \\
 &= 1 + \min \{n\geq 0: {\bf I}(D_{m+1+n})=1 \ \text{and} \ S_{m+1+n}-S_{m+1} > \max_{0\leq i \leq n-1}S_{m+1+i}-S_{m+1},    
\end{align*}
where, by convention, the maximum over empty set is $-\infty$. 

We shift the time origin from $0$ to $m+1$ and the space origin from $0$ to $S_{m+1}$ and let $\widehat{\cal F}_0={\cal F}_{m+1}$ and $\widehat{S}_n = S_{m+1+n}-S_{m+1}$. Further, introduce events $\widehat{D}_n = D_{m+1+n}$ and $\widehat{M}_n = \{\widehat{S}_n > \widehat{S}_j, \ \text{for} \ 0\leq j < n\}$, and let $\widehat{T}_1= \min \{n\geq 0: {\mathbf I}(\widehat{M}_n\cap\widehat{D}_n)=1$. Then the result obtained in {\bf Step 1} for $T_1$ may be applied to $\widehat{T}_1$, too.
Therefore, given $T_k=m$ and for $c>0$,
\begin{align*}
    {\mathbb E} (\exp (c(T_{k+1}-T_k) \ | \ {\cal F}_{T_k})
&=
{\mathbb E}(\exp(c\psi_{m+1})\exp (c\widehat{T}_1) \ | \ \widehat{F}_0, D_m) \\
&\leq
\frac{C_1}{\varepsilon} {\mathbb E} (\exp (c\widehat{T}_1) {\bf I}(D_m)\ | \ \widehat{F}_0) \\
&\leq
\frac{C_1}{\varepsilon} {\mathbb E} (\exp (c\widehat{T}_1) \ | \ \widehat{F}_0)\\
&\leq \frac{C_1C_2}{\varepsilon}\ \ \text{a.s.}
\end{align*}
where the absolute constant $C_1$ comes from Condition 2, $\varepsilon$ comes from \eqref{Phat}, and $C_2$ follows from the {\bf Step 1} of this proof.

Q.E.D. 
\end{proof}

\subsection{Regenerative cycles}\label{appendix:renewal}




We introduce the additional Condition~3 (see below) that states  that  the process exhibits a regenerative structure, in which the increments between successive record times are independent and identically distributed. Conditions~1-3 allow us to establish, see Proposition~\ref{prop3}, a Strong Law of Large Number result, interpreted here as the existence of a limiting drift.
In Appendix~\ref{appendix:conditions1-3forLLP}, we demonstrate that LLPs in Model~\ref{Model2} satisfy Conditions 1--3, making our main results applicable to that framework. We note that Condition~3
significantly simplifies the proofs.

In more detail: in addition to Conditions 1 and 2, we assume \\
\vspace{0.25cm}
{\bf Condition 3.}  
\begin{itemize}\item For any $k\ge 1$,
the distribution of the sequence $\Sigma_k:=\{\xi_{T_k+i}\}_{i\ge 1}$ does not depend on the $\sigma$-algebra ${\cal F}_{T_k}$ and, in particular, on the value of $T_k$. 
\item 
Moreover, all these sequences $\Sigma_k$, $k=1,2,\ldots$ have the same distributions.
\end{itemize}

We can now state an instrumental result for the rest of the development.

\begin{proposition}\label{prop3}
Under the Conditions 1-3, the families of sequences
\begin{equation}\label{iidcycles}
(T_{k+1}-T_k, \{\xi_{T_k+i}, 1\leq i \leq T_{k+1}-T_k\}) \ \ k=1,2\ldots
\end{equation} are a.s. finite and i.i.d. in $k$, and they do not depend on the sequence $\{\xi_i, i\le T_1\}.$\\
Further, their lengths $T_{k+1}-T_k$ and  scopes
$Z_k := \sum_{i=1}^{T_{k+1}-T_k} || \xi_{T_k+i}||$
are i.i.d. and possess finite exponential moments. \\
Furthermore, the random vectors $X_{T{k+1}}-X_{T_k}$, $k\ge 1$ are i.i.d., possess finite exponential moments and do not depend on the random vector $X_{T_1}$. Then ${\mathbb E}(X_{T_2}-X_{T_1}) \neq {\bf 0}$ and the SLLN holds
\begin{equation}\label{newL}
\frac{X_{T_n}-X_0}{n} \to {\mathbb E}(X_{T_2}-X_{T_1}), \quad \text{a.s. and in} \ {\cal L}_1
\end{equation}
and, therefore,
\begin{equation}\label{newL2}
\frac{X_{n}-X_0}{n} \to \frac{{\mathbb E}(X_{T_2}-X_{T_1})}{{\mathbb E}(T_2-T_1)} =: L, \quad \text{a.s. and in} \ {\cal L}_1.
\end{equation}
\end{proposition}

\begin{remark}\label{rem11}  
Note that the direction of vector $L$ may differ from the direction of $l$. However,
the scalar product $\langle L,l \rangle$ must be strictly positive. Indeed, in view of the previous proposition $ \langle L , l \rangle = \lim_n S_n/n \geq \varrho $ where the last inequality follows because of \eqref{G0}.
\end{remark}


\begin{proof}[Proof of Proposition \ref{prop3}] \label{remA3}
This is nothing more than a direct corollary of the previous results and of the basic renewal theory in the presence of additional Condition 3. Indeed, (\ref{newL}) follows since
\begin{equation*}
\frac{X_{T_n}-X_0}{n} =
\frac{\sum_{i=1}^{n-1}(X_{T_{i+1}} -X_{T_i})}{n} + \frac{X_{T_1}-X_0}{n}
\end{equation*}
where the increments in the numerator of the first fraction are i.i.d. with a finite mean and the numerator in the second fraction has a finite mean and does not depend on $n$.
Then (\ref{newL2}) follows from the fact that
$V_k:= \sup_{T_k\leq n\leq T_{k+1}} |X_n-X_{T_k}|$, $k=1,2,\ldots$ are i.i.d. with a finite mean and, therefore,
\begin{equation*}
    \frac{\max_{1\leq k\leq n-1} V_k}{n} \to 0, \quad \text{a.s. and in} \ \  {\cal L}_1 
\end{equation*}
and, further, $V_0:= \sup_{1\leq n\leq T_{1}} |X_n-X_{0}|$ has a finite first moment, too. 
\qed
\end{proof}

\subsection{Regenerative cycles for cones}\label{appendix:renewalcones}


Let ${\mathbb R}^d_+$ denote the positive orthant in ${\mathbb R}^d$, consisting of all vectors with nonnegative coordinates. For the vector $l$ considered in \eqref{G0} let  
\begin{equation}\label{cone1}
  {\cal C} = \{z\in {\mathbb R}^d: \ \langle z,l \rangle > 0\} \ \text{and} \ \widehat{\cal C} =
  {\cal C}\cap {\mathbb R}^d_+. 
\end{equation}


Since $l$ does not need to belong to $\R^d_+$, the set $\widehat{\cal C}$ may be a strict subset of ${\mathbb R}^d_+$. In addition, it can be shown that under the assumptions of Model \ref{Model}, the set $\widehat{\cal C}$ is nonempty, which is of crucial relevance for our main result, see Theorem~\ref{thm:MainResult}.


Let $X_{n,j}$ be the $j$'th coordinate of vector $X_n$. For $0\leq i <j$, introduce the following events:
\begin{align}\label{Eij}
E_{i,j}:= &\left\{ \bigcap_{m=1}^{j-i} 
(X_{i+m}-X_i)\in \widehat{\cal C} \right\} 
\bigcap  \left\{ \min_r (X_{j,r}-X_{i,r})>0 \right\} \bigcap \left\{S_j  >  \max_{i\leq m<j} S_m \right\} 
\end{align}
In particular, if $i=T_k$ and $j=T_{k+1}$, for some $k\ge 1$, then, clearly, 
\begin{align}\label{ETk}
E_{T_k,T_{k+1}} = &
\left\{ \bigcap_{m=1}^{T_{k+1}-T_k}
(X_{T_k+m}-X_{T_k})\in \widehat{\cal C} \right \} 
\bigcap   \left\{ \min_r (X_{T_{k+1},r}-X_{T_k,r})>0 \right\}.
\end{align}

\begin{proposition}\label{prop4}
Assume that 
\begin{itemize}
    \item[(i)] 
Conditions 1-3 hold, 

\item[(ii)] 
The vector $L$ defined by \eqref{newL2} has all coordinates strictly positive: $L=(L_1,L_2,\ldots,L_d)$, with $L_i>0$, for all $i$. 

\item[(iii)]
 ${\mathbb P} (E_{T_1,T_2})>0.$
\end{itemize}
Then
\begin{align}\label{regen1}
    {\mathbb P}
    \left( 
    \bigcap_{m=1}^{+\infty}
\{(X_{T_1+m}-X_{T_1})\in \widehat{\cal C}\}
    \right) >0. 
\end{align}
\end{proposition}

\begin{proof}[Proof of Proposition \ref{prop4}] 
For $c>0$, let 
\[
E^{(k)} := E_{T_k, T_{k+1}}\cap \{ \min_r (X_{T_{k+1},r}-X_{T_k,r})\ge c \}
\]
and choose $c$ such that $z_1:={\mathbb P}(E^{(k)})>0$.

 For $k\geq 1$ and $j=1,2,\ldots,d$, let 
 \begin{align*}
 &U_{k,j} = X_{T_{k+1},j} - X_{T_k,j} \ \text{and}
  \\
&V_{k,j} = \min_{1\leq i \leq T_{k+1}-T_k}
(X_{T_k+i,j}-X_{T_k,j}).    
 \end{align*}
It follows from Proposition \ref{prop3} that, for any $j$, 
the pairs $(U_{k,j},V_{k,j})$ are i.i.d. and possess finite exponential moments.
Since $L_j = {\mathbb E}( U_{k,j}) >0$, we have that  
\[
W_{k,j} := \min \left\{  
V_{k,j}, \inf_{m\ge 0} \left( 
\sum_{i=0}^{m} U_{k+i,j} + V_{m+1,j}
\right) \right\}
\]
is a.s. finite too (and possesses a finite exponential moment). Indeed,
$W_{k,j} \leq V_{k,j}$ and 
\begin{align*}
W_{k,j} \geq & 
\min \left\{  
V_{k,j}, \inf_{m\ge 0} \left( 
(m+1)b + V_{m+1,j}
\right) \right\} + \\
&\min \left\{0, \inf_{m\ge 0}
\sum_{i=0}^m (U_{k+i,j}-(m+1)b)\right\}   
\end{align*}
where $b\in (0,L_j)$ is any constant.
The infimum in the first term in the right-hand side is finite because i.i.d. random variables $V_{m,j}$ has a finite mean. 
The infimum in the second term in the right-hand side in the infimum of partial sums of i.i.d. random variables with a positive mean which is finite a.s. 

Thus, there exists $y>0$ such that $z_2:= {\mathbb P} (\min_jW_{k,j}>-y)>0$.

Choose $k$ such that $(k-1)c >y$.
Then
\begin{align*}
  & {\mathbb P}
    \left( 
    \bigcap_{m=1}^{+\infty}
\{(X_{T_1+m}-X_{T_1})\in \widehat{\cal C}\}
    \right) \geq
    \left( \prod_{i=1}^{k-1} {\mathbb P}(E^{(i)}) \right)  {\mathbb P}(\min_j W_{k,j}>-y) =
    z_1^{k-1}z_2 >0.
\end{align*} \qed
\end{proof}

\begin{cor}\label{corA3}
Assume that the conditions of Proposition \ref{prop4} hold.
Assume further that the initial conditions are such that the event $E_{0,n}$ has a strictly positive probability, for some $n\ge 1$. 
 Then
 \begin{align}\label{eqcorA3}
   {\mathbb P}
    \left( 
    \bigcap_{n=1}^{+\infty}
\{(X_{n}-X_{0})\in \widehat{\cal C}\}
    \right) >0.   
 \end{align}
\end{cor}
\begin{proof}[Proof of Corollary \ref{corA3}]
Choose the smallest $n$ such that ${\mathbb P}(E_{0,n})>0$. Denote by $\widehat{D}_n$ the event
\begin{align}\label{hatMn}
\widehat{D}_n = \bigcap_{i =  1}^{+\infty}
\{X_{n+i}-X_n \in \widehat{\cal C}\}.
\end{align}
Then 
the probability in (\ref{eqcorA3}) is equal to
\begin{align*}
    &{\mathbb P}(E_{0,n}\cap \widehat{D_n}) =
    {\mathbb P} (E_{0,n}) {\mathbb P}(D_n \cap M_n) {\mathbb P}(\widehat{D}_n \ | \ D_n\cap M_n)>0   
\end{align*}
because  all three multipliers are strictly positive. \qed
\end{proof}

If the event $D_n\cap M_n$ occurs, then
$n\in \{T_k\}_{k\ge 1}$.
Therefore, under the Conditions 1-3, the probability 
\begin{equation} 
    \widehat{p}:= {\mathbb P}(\widehat{D}_n \ | \ D_n\cap M_n) >0
\end{equation}
does not depend on $n$. 

Let $1\leq \nu_1 < \nu_2 < \ldots$ be the consecutive indices $k$ such that events $\widehat{D}_{T_k}$ occur. 

\begin{proposition}\label{prop5}
Assume that the conditions of Proposition \ref{prop4} hold. Then,
\begin{itemize}
    \item The 
random elements
\begin{equation}\label{elements}
\left( \nu_{k+1}-\nu_k, T_{\nu_{k+1}}-T_{\nu_k},
\{X_{T_{\nu_k+i}}-X_{T_{\nu_k}}, 1\leq i \leq 
T_{\nu_{k+1}}-T_{\nu_k} \} \right)
\end{equation}
 are i.i.d. and do not depend on $\left(\nu_1, T_{\nu_1}, \{X_i, 0\leq i \leq T_{\nu_1}\} \right)$;
\item
further, the random variables $\nu_1$ and $\nu_{k+1}-\nu_k$, $k\ge 1$ 
have finite exponential moments, and therefore $T_{\nu_1}$ and $T_{\nu_{k+1}}-T_{\nu_k}$, $k\ge 1$ have finite exponential moments, too.
\end{itemize}
\end{proposition}

\begin{proof}[Proof of Proposition \ref{prop5}].
It is enough to prove that $\nu_2-\nu_1$ has an exponential moment. 

Let $A^c$ denote the complement of event $A$. We have that, for any $k\leq m$,
\begin{align*}
{\mathbb P} (\nu_2-\nu_1 > n \ | \ \nu_1=k, T_k=m, {\cal F}_m) 
&= {\mathbb P} 
\left( \bigcap_{i=2}^{n+1} \widehat{D}^c_{T_i}
\ | \ \nu_1=1, T_1=m \right) \\
&\leq 
\frac{ {\mathbb P}\left( 
 \bigcap_{i=2}^{n+1} \widehat{D}^c_{T_i}\right)}{
 {\mathbb P}(\nu_1=1, T_1=m)}
= (1-\widehat{p})^n/q
\end{align*}
where we choose $m$ such that
$u:={\mathbb P}(T_1=m)>0$, and then
$q:= {\mathbb P}(\nu_1=1, T_1=m) =
u\widehat{p}>0$. \qed
\end{proof}

\begin{cor}\label{corA4}
Assume that the initial environment ${\cal F}_0$ and the jump distributions up to time $n$ are arbitrary. Assume further that the conclusions of Proposition \ref{prop3} hold. 
Assume there is (the smallest) $n$ such that the event $E_{0,n}$ has a positive probability.
Then, 
given the event $E_{0,n}\cap\widehat{D}_n$ occurs, the statement of Proposition \ref{prop5} continues to hold.
\end{cor}

\section{ Conditions 1-3 for LLPs in Model \ref{Model2}}\label{appendix:conditions1-3forLLP}


In this section, we show that any LLP in Model 2 satisfies Conditions 1-3. As a consequence, all results from Appendix  \ref{appendix:mathematicalframework} are applicable to LLPs in Model \ref{Model2} . Using these findings, we can conclude the proofs of Theorem \ref{thm:learningdrift} (Appendix \ref{sec:proofthm}) and Theorem \ref{thm:MainResult} (Appendix \ref{sec:proofMainresult}).

For LLPs in Model \ref{Model2} we consider $\mathcal{G} = \{ \mu _r, \mu_g\}$ and the filtration defined by $\mathcal{F}_n = \sigma ( X_0, A_0, \dots,X_n,A_n)$. Then $(X_n)_{n \geq 0}$ satisfies Condition 1 since for $S_n ^{(l)} = \langle X_n , l \rangle $ we obtain using \eqref{eq:model2 2} that
\begin{align*}
    \varrho: = &\inf_{n\geq 0} \mathbb{E} ( S_{n+1} ^{(l)} - S_{n} ^{(l)} \ | \ \mathcal{F}_n )  \\
    &=\inf_{n \geq 0} \langle d(\mu_{A_n} ) , l \rangle \\
    & =\min \{ \langle d(\mu_r), l \rangle, \langle d(\mu_g), l \rangle \} > 0,
\end{align*}
which is exactly \eqref{G0}.

Similarly for Condition 2 we have that $\xi_n = X_{n+1} - X_n$ given $\mathcal{F}_n$ has distribution $\mu_{A_n}$ then \eqref{eq: model2 3} is exactly \eqref{G4}. Proposition \ref{prop: finite Tk} can be applied to $(X_n)_{n \geq 0} $ and implies that the random times $(T_k)_{k\geq0}$ given by \eqref{T1} and \eqref{Tk} are finite a.s. 

Finally Condition 3 holds due to the following lemma.

\begin{lemma}\label{prop: local is renewal}
    For all $x_1,\dots,x_N \in \mathbb{Z}^2$ 
    \begin{align*}
       &\mathbb{P} ( X_{T_k + 1} - X_{T_k} = x_1, \dots, X_{T_k + N} - X_{T_k} = x_N ) = \mathbb{P}( X_1 = x_1, \dots, X_N = x_N \ | \ \mathbf{I}(D_0) = 1 ). 
    \end{align*}
\end{lemma}

\begin{proof}    
For the sake of formality, we temporarily adopt the following notation, making explicit the environment dependence on the state-action path:
\begin{equation}
    \beta_{N+1} ( x , a ) [x_0,a_0,x_1,a_1,\dots,a_N,x_{N+1}].
\end{equation}
We start with the following observation: the environment at time $N $ left by an agent starting from $x_0$ is the same as the translation by $x_0$ of the environment at time $N$ left by an agent starting from $0$, explicitly
\begin{align}
        &\beta_{N+1} ( x , a ) [x_0,a_0,x_1,a_1,\dots,a_N,x_{N+1}]  \notag \\ 
                & = \beta_{N+1} ( x + x_0, a ) [0,a_0,x_1-x_0,a_1,\dots,a_N,x_{N+1}] \label{lemma:shift env}
    \end{align}
    We can still be more precise about how the environment depends on the state-action sequence.
    Definition~\ref{defn:learningprocess} indicates that the environment at a given state-action pair is determined solely by the past behavior of the process at that state.
More precisely,  consider the indices $0\leq i_1 <\dots < i_M \leq N$ such that $ (x_{i_j},a_{i_j}) = (x,a)$. Then
\begin{align}
    &\beta_{N+1} ( x , a ) [x_0,a_0,x_1,a_1,\dots,a_N,x_{N+1}] \notag \\ 
                &= \beta_{M+1} ( x , a ) [x_{i_1},a_{i_1},x_{i_1+1},a_{i_2},\dots,x_{i_{M}},a_{i_M},x_{i_M+1}] \label{lemma: environment independence}
\end{align}

With \eqref{lemma:shift env} and \eqref{lemma: environment independence} we prove Lemma \ref{prop: local is renewal} as follows. Fix $x_0 \in \mathcal{S}$, and consider $T \geq 1$ and $x_1,\dots,x_T \in \mathbb{Z}^2$ and $a_0,\dots,a_{T-1} \in \mathcal{A}$ such that the event $\Omega ' = \{  T_k = T, A_0 = a_0, X_1 = x_1, \dots, A_{T-1} = a_{T-1}, X_T = X_T \}$ satisfies $ \mathbb{P} (\ \Omega ' \ | \ X_0 = x_0 \ ) > 0. $

Now fix also $N \geq 1$ and $\xi_1,\dots,\xi_N \in \mathbb{Z}^2$ and $b_0,\dots,b_{N-1} \in \mathcal{A}$. Assume that the path defined by $\xi_1 \to \dots \to \xi_N$ belongs to the nonnegative quadrant $\R^2_{+}$. It suffices to show that:
\begin{align}
    &\mathbb{P} ( A_{T} = b_0, X_{T+1} - x_T = \xi_1,\dots, A_{N-1} = b_{N-1}, X_{T+N} - x_T = \xi_N \ | \ \Omega ' \ ) \notag \\
    & = \mathbb{P} ( A_0 = b_0, X_{1} = \xi_1,\dots, A_{N-1} = b_{N-1}, X_{N} = \xi_N \ | \ X_0 = 0, \ \mathbf{I}(D_0) = 1 ). \notag
\end{align}
This easily follows by induction. \qed

\end{proof}


Therefore we have obtained that any LLP $(X_n,A_n,\beta_n)_{n\geq 0}$ in Model \ref{Model2} satisfies Conditions 1-3. In particular Proposition \ref{prop: finite Tk} and \ref{prop3} hold for $(X_n)_{n\geq 0}$. 

\section{Proof of Theorem \ref{thm:learningdrift}} \label{sec:proofthm}

Denote by $N_{k}^{(r)}$ the number of times within the time interval $(T_k,T_{k+1}]$ when action $r$ is taken and by $N_k^{(g)}$ the number of times when action $g$ is taken. Clearly, $N_k^{(r)}+N_k^{(g)}=T_{k+1}-T_k$.
\\
It follows from Proposition \ref{prop5} that the random variables $\{N_k^{(r)}\}_{k\geq 1}$ are i.i.d. with a finite mean (and, in fact, with finite exponential moments). Therefore, the SLLN holds:
\begin{equation*}
\frac{1}{m} \sum_{k=1}^m N_k^{(r)} \to {\mathbb E} N_1^{(r)} \ \text{a.s. and in} \ 
{\cal L}_1.
\end{equation*}
Then the SLLN for compound renewal processes implies that

\begin{align*}
   \lim _{n \to \infty} \frac{1}{n} \sum_{i=1}^n {\mathbb I} (A_i=r) 
   & = \frac{1}{{\mathbb E} (T_2-T_1)}  \lim_{m \to \infty } \frac{1}{m} \sum_{k=1}^m N_k^{(r)}  \\
   & = \frac{{\mathbb E} N_1^{(r)}}{{\mathbb E} (T_2-T_1)} =:\alpha >0, \ \ \text{a.s. and in} \ {\cal L}_1. 
\end{align*}
\qed

\subsection{The average number of times the agent applies red action}

Now we will obtain upper and lower bounds for $\alpha$, the average number of times the agent applies red action, defined by \eqref{alpha1}. In particular, since the limiting drift $L$ given by \eqref{newL2} is characterized in Theorem \ref{thm:learningdrift} as $L = \alpha d(\mu_r) + (1-\alpha)d(\mu_g)$ this bounds are aimed to find conditions for $L$ to be in the positive quadrant.

\begin{lemma}\label{cor:lowerbound}
    For any LLP in Model \ref{Model2} consider $q = \varphi ( \ r \ | \ \beta_0 ( x, \cdot ) )$, $\varrho = \min\{ \langle d (\mu_r ), l \rangle , \langle d (\mu_g ), l \rangle \}$ and $\alpha$ given by Theorem \ref{thm:learningdrift}. Then \begin{equation}\label{eq:alphabound}
        1-(1-q)\varrho \geq \alpha \geq  q \varrho 
    \end{equation}
\end{lemma}

\begin{proof}
Let $\widetilde{\alpha}$ be the frequency of times when the process visits a new state (i.e. a state that was not visited before), namely
\begin{align*}
    \widetilde{\alpha} = \lim_{n\to\infty} \frac{1}{n} 
    \sum_{i=1}^n {\mathbf I} (X_i \ \ \text{visits a new state}).
\end{align*}
This limit exists a.s. and in ${\cal L}_1$, that it may be shown completely analogously to that for $\alpha$.
Then
\begin{align}\label{aliphar}
    \alpha = q \widetilde{\alpha} + \widehat{\alpha} (1-\widetilde{\alpha}) \geq q \widetilde{\alpha},
\end{align}
where $\widehat{\alpha}$ is the frequency of times for choosing action $r$ when visiting an ``old'' state (the existence of $\widehat{\alpha}$ follows directly). 
Switching from action $r$ to action $g$, we arrive at the inequality
\begin{align}\label{alphag}
    1-\alpha = (1-q)\widetilde{\alpha} + (1-\widehat{\alpha}) (1-\widetilde{\alpha}) \geq 
    (1-q)\widetilde{\alpha}.
\end{align}

Now, Conditions 1 and 2 from Appendix A imply that, for $\varrho = \min \{\langle d(\mu_r),l \rangle , \langle d(\mu_g), l \rangle \}$,
\begin{align*}
    \liminf_{n\to\infty} \frac{S_n}{n} \geq \varrho , \quad \text{a.s.}
\end{align*}
Since all random variables $\{S_n/n\}$ take values between $-1$ and $1$, they are uniformly integrable and, therefore,
\begin{align*}
    \liminf_{n\to\infty} \frac{{\mathbb E}(S_n)}{n} \geq \varrho.
\end{align*}

Since the jumps of the processes are of size $1$ only and $l$ is a unit vector, for any $n$ and $z$ the inequality $S_n>z$ implies that the number of visited new states
by time $n$ is at least $z$. Therefore, $\widetilde{\alpha} \geq \varrho$ and then $\alpha \geq q\varrho$ and
$1-\alpha \geq (1-q)\varrho$.
\qed
\end{proof}

\section{Proof of Theorem \ref{thm:MainResult}} \label{sec:proofMainresult}

 Let $\MM = (\SS,\A,p)$ be as in Model \ref{Model}. Consider $0< \alpha_0 < \alpha_1 \leq 1$ such that $\delta d_r + (1-\delta)d_g \in \R^2_+$ for all $\delta \in (\alpha_0,\alpha_1)$. Assumption \eqref{parameter5} implies that $0 < q_0 := \frac{\alpha_0}{\varrho} < q_1: = \frac{\alpha_1+ \varrho - 1}{\varrho} \leq 1$. Then, for $q \in (q_0,q_1)$, we have 
\begin{equation}\label{eq:bounds q}
    \alpha_0 < q \varrho \ \text{ and } \ 1 - (1-q)\varrho < \alpha_1
\end{equation}

    Fix any LLP in $\MM$ with $q := \varphi ( r \ | \ \beta_0 ( s, \cdot ) ) \in (q_0,q_1)$ and consider its associated free process (see Remark \ref{rmk:freeprocess}). This free process is an LLP in Model \ref{Model2}, so that as shown in Appendix \ref{appendix:conditions1-3forLLP} it fits the framework of Appendix \ref{appendix:mathematicalframework}. In particular, thanks to Lemma \ref{cor:lowerbound} and the inequalities in \eqref{eq:bounds q},  the free process has a limiting drift $L$ with strictly positive coordinates. Then the results of Appendix \ref{appendix:renewalcones} (and, in particular, Lemma \ref{corA3}) may be applied to the free process.

We first establish the existence of an almost surely finite time after which the state process and a free process, started at that time, couple. Let 
\begin{align*}
    &T_{\nu_1} = \min \{ n\geq 1: X_{n,1}>0, X_{n,2}>0, X_{n,1}+X_{n,2} > \max_{k<n} (X_{k,1}+X_{k,2}), {\mathbf I}(\widehat{D}_n)=1 \}.
\end{align*}
We show that this time is controlled by a geometric sum of almost surely finite random variables.
Set  \begin{align*}
&\theta_1 = \min \{n\ge 1:   X_{n,1}>0, X_{n,2}>0,  X_{n,1}+X_{n,2} > \max_{k<n} (X_{k,1}+X_{k,2})\}.  
\end{align*}
By the irreducibility of the process $(X_n)_{n\geq 0}$, the time $\theta_1$ is almost surely finite.
If the event $\widehat{D}_{\theta_1}$ occurs, then $T_{\nu_1} = \theta_1$.
If not, we introduce two additional random times:
\begin{equation*}
 \kappa_1 = \min \{m\ge 1: X_{\theta_1+m} -X_{\theta_1} \notin \widehat{\cal C} \}  
\end{equation*}
and then
\begin{align*}
    &\theta_2 =
   \min \{n\ge \theta_1+\kappa_1:   X_{n,1}>0, X_{n,2}>0, X_{n,1}+X_{n,2} > \max_{k<n} (X_{k,1}+X_{k,2})\}.  
\end{align*}
 If $\widehat{D}_{\theta_1}$ does not occur, then $\kappa_1$ must be finite. Then, again by the irreducibility, $\theta_2$ is a.s. finite, too, and we have $T_{\nu_1} = \theta_2$ and the claim follow. The induction arguments complete the proof of finiteness of $T_{\nu_1}$. Note that our reasoning is based on Lemma \ref{corA3}, that implies that there is $c >0$ such that
\begin{equation*}
    \P \left( \mathbf{I} (\widehat{D}_{\theta_{k+1}} ) = 1 \ | \ \mathbf{I} (\widehat{D}_{\theta_{k}} ) = 0 , \F_{\theta_{k+1}} \right) \geq c > 0, \quad \text{for all } k \geq 1 . 
\end{equation*}

From the finiteness of $T_{\nu_1}$ we derive the transience of the process $(X_n)_{n\geq 0}$. We now prove \eqref{eq:Loss function}. 
{\bf Proof of the upper bound.} Since in Model 1 all increments $|X_{n+1}|-|X_n|$ cannot be bigger than 1, we get 
$J_{\gamma}\leq 1/(1-\gamma)$ a.s.

{\bf Proof of the lower bound.} Write, for short, $\Delta_i := c(X_i,X_{i+1})= |X_{i+1}|-|X_i|$ (where, recall, $|\cdot|$ the $l_1$ norm), and let 
\begin{equation*}
    {\mathbb E} \sigma_1(\gamma) := {\mathbb E} \left(\sum_{i=0}^{T_{\nu_1}-1} \gamma^i \Delta_i \right).
\end{equation*}
Then 
\begin{align*}\label{discount}
  {\mathbb E} \left(\sum_{i=0}^{+\infty} 
  \gamma^i \Delta_i \right)  &= {\mathbb E} \sigma_1(\gamma)+ \sum_{k=1}^{+\infty}
  {\mathbb E} 
  \left(\sum_{i=0}^{T_{\nu_{k+1}-1}-T_{\nu_k}}
  \gamma^{T_{\nu_k}+i}\Delta_{T_{\nu_k+i}}
  \right)\\
  &={\mathbb E} \sigma_1(\gamma)+
  \sum_{k=1}^{+\infty}
  {\mathbb E} \gamma^{T_{\nu_k}} \cdot 
  {\mathbb E} 
  \left(\sum_{i=0}^{T_{\nu_{k+1}}-T_{\nu_k}-1}
  \gamma^{i}\Delta_{T_{\nu_k+i}}
  \right)\\
  &= {\mathbb E} \sigma_1(\gamma)+\sum_{k=1}^{+\infty}
  {\mathbb E} \gamma^{T_{\nu_k}} \cdot 
  {\mathbb E} 
  \left(\sum_{i=0}^{T_{\nu_{2}}-T_{\nu_1}-1}
  \gamma^{i}\Delta_{T_{\nu_1+i}}
  \right)\\
  &=
  {\mathbb E} \sigma_1(\gamma)+\frac{{\mathbb E} \gamma^{T_{\nu_1}}}{1-{\mathbb E}\gamma^{T_{\nu_2}-T_{\nu_1}}} 
  {\mathbb E}
  \left(\sum_{i=0}^{T_{\nu_2}-T_{\nu_1}-1} \gamma^i\Delta_{T_{\nu_1+i}}\right) \\
  &\equiv {\mathbb E} \sigma_1(\gamma) + \frac{{\mathbb E} \gamma^{T_{\nu_1}}}{1-{\mathbb E}\gamma^{T_{\nu_2}-T_{\nu_1}}} {\mathbb E} \sigma_2(\gamma).
\end{align*}

Here the equality in the second line follows from the mutual independence of regenerative cycles. Now note that: \\
(i)  $\sigma_1(\gamma) +  |X_{0}|  \ge 0$  for any $\gamma\in (0,1]$ since 
\begin{align*}
   |X_0|+ \sum_{i=0}^{T_{\nu_1}-1} \gamma^i \Delta_i =  
   |X_1|+ \sum_{i=1}^{T_{\nu_1}-1} \gamma^i \Delta_i = 
   \gamma^{T_{\nu_1}} |X_{T_{\nu_1}}| 
+ \sum_{i=1}^{T_{\nu_1}-1} (1-\gamma) \gamma^{i-1} |X_{i}| \geq 0;
\end{align*}
(ii) $\sigma_2(\gamma)$ is positive by similar reasons;
\\
(iii) using the results from the previous Section, one can conclude that $\sigma_2(\gamma)$ tends a.s. and in ${\cal L}_1$ to $|X_{T_{\nu_{k+1}}}|-|X_{T_{\nu_k}}|$, which is strictly positive and has a finite mean;\\
(iv) random variables $\{ \gamma^{T_{\nu_1}}\}$ are uniformly integrable and, therefore, ${\mathbb E} \gamma^{T_{\nu_1}} \uparrow 1$, as $\gamma \to 1$;
\\
(v)
since $T_{\nu_2}-T_{\nu_1}$ has an exponential moment,
\begin{align*}
 1- {\mathbb E} \gamma^{T_{\nu_2}-T_{\nu_1}} \sim (1-\gamma) {\mathbb E}  
 (T_{\nu_2}-T_{\nu_1}), \text{ as } \gamma \to 1.
\end{align*}

All in all, 
\begin{align*}
   {\mathbb E} \left(\sum_1^{\infty} 
  \gamma^i \Delta_i \right) \geq
  \frac{C}{1-\gamma} - |X_0| \geq 
  \frac{C'}{1-\gamma}
\end{align*}
for $\gamma$ close enough to $1$,
where $C$ and $C'$ are strictly positive and finite constants. \qed

\section{Model~\ref{Model}}\label{appendix:loadbalancing}

This section is devoted to the proofs of Lemma~\ref{lemma:lyapunovgreen}, Lemma~\ref{lemma:loadbalancing.},  Lemma~\ref{lemma:serverallocation.} and Lemma~\ref{rmk: O(1)}.

\textbf{Proof of Lemma \ref{lemma:lyapunovgreen}}
We consider an appropriated Lyapunov function. Take $v \in \R^2_+$ as in \eqref{parameter2}, then let $V(x) = \langle x , v \rangle ^2$ and $F(x) = \langle x , v \rangle$. We claim that there are $B,c>0$ s.t. for all $x \in \Z^2_{+}$
\begin{equation}\label{eq:lyapunov}
    \E( V(X_1) - V(X_0) \ | \ X_0 = x ) \leq -c F(x) +B
\end{equation} 
Thus, Lemma~\ref{lemma:lyapunovgreen} follows from \eqref{eq:lyapunov} and \cite[Theorem~14.3.7]{meyntweedie}.

We derive \eqref{eq:lyapunov} in the following way. Observe that $V(y) - V(x) = \langle y + x , v \rangle \langle y-x, v \rangle$, from this we conclude \begin{align*}
    \E( V(X_1) - V(X_0) \ | \ X_0 = x ) 
    &= \E ( \langle X_1+x,v \rangle \langle X_1-x,v \rangle \ | \ X_0 = x ) \\
    & = \sum_{\xi} \langle 2x + \xi , v \rangle \langle \xi , v \rangle p(x + \xi \ | \ x , g )\\
    & = 2F(x) \E( \langle X_1 - X_0 , v \rangle \ | \ X_0 = x ) + \sum_{\xi} \langle \xi , v \rangle ^2 p(x + \xi \ | \ x , g ) \\
    & \leq -2b F(x) + B
\end{align*} 
Where $b,B>0$ are such that $\E( \langle X_1 - X_0 , v \rangle \ | \ X_0 = x ) < -b$ for all $x \in \Z^2_{+} \setminus \{(0,0)\}$ (note that $F((0,0))=0)$ ) and $ \sum_{\xi} \langle \xi , v \rangle ^2 p(x + \xi \ | \ x , g ) < B$ for all $x \in \Z^2_{+}$.
\qed

 \textbf{Proof of Lemma \ref{rmk: O(1)}.}
    We start by showing that the total cost is finite, i.e, $J^g(x) = \E(|X_\infty ^g|) - |x| < + \infty$ where $(X_n ^g)_{n \geq 0}$ is the process defined by an agent always making ``green'' action. Since by \eqref{parameter2} there is $v \in \R^2_+$ such that $\E( \langle X_1^g - x , v \rangle \ | \ X_0^g = x ) < -b$ for some $b > 0$ and $x \neq (0,0)$. We can take this $v$ to further satisfy $|x| = \langle x, (1,1) \rangle \leq \langle x , v \rangle$ for all $x \in \Z^2_{+}$, so that by Lemma \ref{lemma:lyapunovgreen} we obtain $\E(|X_\infty ^g|) \leq \E(\langle X_\infty^g, v \rangle ) < + \infty$.

We now show \eqref{eq: O(1)}, this is that for discounted cases, the ``green'' policy gives an order 1 cost.

 For simplicity denote $Z_i = |X_i^g|$ from where $c(X_{i+1}^g,g,X_i^g) = Z_{i+1}-Z_i$.
    We obtain
    \begin{align*}
    J_\gamma^g(X_0) &:= 
    \lim_N \E \left( \sum_{i=0} ^{N} \gamma^i (Z_{i+1}-Z_i) \right) \\
&=  \lim_N \E \gamma^{N} |X_N^g| -  |X_{0}^g| + \lim_N \sum_{i=0}^{N} (1-\gamma) \gamma^i \E |X_{i+1}^g|
\end{align*}
Using the convergence of the first moment of $X_\infty^g$, we obtain
that 
  \begin{equation*}
    J_\gamma(X_0) \le C(X_0) + \E(|X_\infty^g|),
\end{equation*}
where $C(X_0)$ is a positive constant independent of $\gamma$ and hence  $J_\gamma(X_0)$ is of order 1. \qed

\textbf{Proof of Lemma \ref{lemma:loadbalancing.} } We first verify that $(\SS,\A,p)$ is an instance of Model~\ref{Model}. We have $\SS = \Z^2_{+}$, $\A = \{r,g\}$ and the transitions probabilities are given explicitly by:
\begin{enumerate}    
    \item \textbf{Interior states } $x_1 > 0$, $x_2 > 0$:\begin{align*}
            &p(x+(1,0) \mid x, a) = \dfrac{\lambda p_a}{\lambda + \mu_1 + \mu_2} \\[0.5em]
            &p(x+(0,1) \mid x, a) = \dfrac{\lambda (1-p_a)}{\lambda + \mu_1 + \mu_2} \\[0.5em]
            &p(x+(-1,0) \mid x, a) = \dfrac{\mu_1}{\lambda + \mu_1 + \mu_2} \\[0.5em]
            &p(x+(0,-1) \mid x, a) = \dfrac{\mu_2}{\lambda + \mu_1 + \mu_2}
    \end{align*} 
    \item \textbf{Boundary states with } $x_2 = 0$, $x_1 > 0$:    \begin{align*}    &p(x+(1,0) \mid x, a) = \dfrac{\lambda p_a}{\lambda + \widetilde{\mu}} \\[0.5em]
            &p(x+(0,1) \mid x, a) = \dfrac{\lambda (1-p_a)}{\lambda + \widetilde{\mu}} \\[0.5em]
            &p(x+(-1,0) \mid x, a) = \dfrac{\widetilde{\mu}}{\lambda + \widetilde{\mu}}
    \end{align*} 
    \item \textbf{Boundary states with } $x_1 = 0$, $x_2 > 0$:    \begin{align*}
    &p(x+(1,0) \mid x, a) = \dfrac{\lambda p_a}{\lambda + \widetilde{\mu}} \\[0.5em]
            &p(x+(0,1) \mid x, a) = \dfrac{\lambda (1-p_a)}{\lambda + \widetilde{\mu}} \\[0.5em]
            &p(x+(0,-1) \mid x, a) = \dfrac{\widetilde{\mu}}{\lambda + \widetilde{\mu}}
    \end{align*}
    \item \textbf{Origin} $x = (0,0)$:\begin{align*}
            &p((1,0) \mid (0,0), a) = p_a \\[0.5em]
            &p((0,1) \mid (0,0), a) = 1 - p_a
        \end{align*}
\end{enumerate}

These equations define $\mu_a,\mu_a', \mu_a '' \in \PP$ as in \eqref{eq: transitionsmodel1} of Model \ref{Model}. We verify then that \eqref{parameter2}, \eqref{parameter4} and \eqref{parameter5} hold. For this we compute explicitly the drifts:
\begin{align*}
      &d_a \defeq d (  \mu_a ) = \left( \frac{\lambda p_a - \mu_1}{\lambda + \mu_1 + \mu_2}, \frac{\lambda (1 - p_a) - \mu_2}{\lambda + \mu_1 + \mu_2}\right),\label{eq:loadbalancing drift}\\
      &d_a ' \defeq d (\mu_a ') = \left( \frac{\lambda p_a - \widetilde{\mu}}{\lambda + \widetilde{\mu}}, \frac{\lambda (1 - p_a)}{\lambda + \widetilde{\mu}}\right) \text{ and } \\
      &d_a '' \defeq d (\mu_a '') = \left( \frac{\lambda p_a}{\lambda + \widetilde{\mu}}, \frac{\lambda (1 - p_a)-\widetilde{\mu}}{\lambda + \widetilde{\mu}}\right).
\end{align*}

Our analysis of \eqref{parameter2} starts by examining the slopes of the lines perpendiculars to the vectors $d_g$, $d_g'$, and $d_g''$ respectively. We conclude that for \eqref{parameter2} to hold we need 
 \begin{align*}
      &\frac{\widetilde{\mu}-\lambda p_g}{\lambda(1-p_g) } >\max \left\{ \frac{\lambda p_g - \mu_1}{\mu_2-\lambda(1-p_g)} ,\frac{\lambda p_g}{\widetilde{\mu}-\lambda(1-p_g)}\right\}
 \end{align*} 
This follows from \eqref{eq:loadbalancing1}, \eqref{eq:loadbalancing2}, and \eqref{eq:loadbalancing3}.

Consider now \eqref{parameter4}. Under condition \eqref{eq:loadbalancing2} with $l = e_1$, the quantity $\varrho$ can be computed explicitly
\begin{align*}
    \varrho &\defeq \min \{ d_r \cdot e_1, d_g \cdot e_1 \} = \\
    & = \min \left\{ \frac{\lambda p_r - \mu_1}{\lambda + \mu_1 + \mu_2} , \frac{\lambda p_g - \mu_1}{\lambda + \mu_1 + \mu_2} \right\} = \frac{\lambda p_r - \mu_1}{\lambda + \mu_1 + \mu_2}>0,
\end{align*}
Next, for \eqref{parameter5}, condition \eqref{eq:loadbalancing1} allows us to take $\alpha_1 = 1$ and 
\begin{align*}
    \alpha_0 &= \frac{\langle d_g, e_2 \rangle }{\langle d_r - d_g, e_2 \rangle } = \frac{\lambda(1-p_g)-\mu_2}{\lambda(p_g-p_r)}.
\end{align*}
We conclude that 
\begin{align*}
    \alpha_0 < \varrho \iff  \mu_2 < \frac{(\lambda p_r - \mu_1)(p_g-p_r)}{\lambda+\mu_1+\mu_2} + \lambda (1-p_g) 
\end{align*}
where the right-hand side is precisely \eqref{eq:loadbalancing3}.

Therefore, the load balancing example is an instance of Model~\ref{Model}.

Finally, we note that Theorem \ref{thm:MainResult} can be applied to the process defined by an agent always applying red action, implying transience for this process. Alternatively one can directly apply Theorem 3.3.1 of \cite{constructivetheoryofmarkovchains} to show transience taking into account \eqref{eq:loadbalancing1}.\qed

\textbf{Proof of Lemma \ref{lemma:serverallocation.}.} We follow the same steps as in the proof of Lemma \ref{lemma:loadbalancing.}. The transitions $p ( \ y \dado x, a)$ are obtained similarly as in the previous load balancing example. In particular the drifts obtained are:
\begin{align*}
      &d_r \defeq d (  \mu_r ) = \left( \frac{\lambda-\mu }{2\lambda + \mu}, \frac{\lambda}{2\lambda + \mu }\right) \\
      &d_g \defeq d (  \mu_g ) = \left( \frac{\lambda }{2\lambda + \mu}, \frac{\lambda - \mu}{2\lambda + \mu }\right)  \ \ \\
      &  d_a ' \defeq d (\mu_a ') = \left( \frac{\lambda - \widetilde{\mu}}{2\lambda + \widetilde{\mu}}, \frac{\lambda }{2\lambda + \widetilde{\mu}}\right) \ \ \text{ for } a=r,g\\
      &d_a '' \defeq d (\mu_a '') = \left( \frac{\lambda}{2\lambda+ \widetilde{\mu}}, \frac{\lambda - \widetilde{\mu}}{2\lambda + \widetilde{\mu}}\right) \ \ \text{ for } a=r,g
\end{align*}

    We check \eqref{parameter2}. Indeed, in this example we have that the region $R_a = \{v \in \R^2_+ : \langle w, v \rangle < 0 \text{ for } w = d_a,d_a',d_a'' \}$ is non empty for $a=r,g$ from where the conclusion of Lemma \ref{lemma:lyapunovgreen} will hold for both actions $a=r,g$. As in the proof of Lemma \ref{lemma:loadbalancing.} we consider the slopes of the lines perpendicular to the vectors $d_a,d_a',d_a''$, from where we see that $R_a$ is non empty for $a=r,g$ because of \eqref{eq:singleserver1b} and \eqref{eq:singleserver2b}.
    

Next for \eqref{parameter4} we can consider
\begin{equation}
    \alpha_0 := \frac{\mu - \lambda}{\mu} < \alpha_1 := \frac{\lambda}{\mu}.
\end{equation}
where the inequality because \eqref{eq:singleserver1b}.

We conclude by showing that \eqref{parameter5} hold. First, \eqref{eq:singleserver1b} implies that $l = \frac{1}{\sqrt{2}}(1,1)$ is a common direction for both drifts: $\langle d_r, l \rangle = \langle d_g, l \rangle > 0$. Indeed $\varrho := \min \{ \langle d_r, l \rangle ,  \langle d_g, l \rangle \} = \frac{1}{\sqrt{2}}\frac{2\lambda - \mu}{2\lambda + \mu }>0$, and then \eqref{eq:singleserver3b} implies \eqref{parameter5}. \qed

\end{document}